\documentclass{article}
\usepackage{graphicx} 
\usepackage{amsmath,mathrsfs}
\usepackage{amsfonts}
\usepackage{graphicx}
\usepackage{amsthm}
\usepackage[utf8]{inputenc}
\usepackage{graphicx}
\usepackage{float}
\usepackage{bbm}
\usepackage{bm}
\usepackage{geometry}
\usepackage{setspace}
\usepackage{enumerate}
\usepackage{verbatim}
\usepackage{color,extarrows}
\usepackage[breaklinks]{hyperref}
\hypersetup{hidelinks}
\usepackage[numbers]{natbib}

\let\oldcite\cite
\renewcommand{\cite}[1]{\oldcite{#1}}

\newcommand{\R}{\mathbb{R}}
\newcommand{\E}{\mathbb{E}}
\newcommand{\Q}{\mathbb{Q}}

\newcommand{\D}{\mathrm{d}}
\newcommand{\mL}{\mathcal{L}}

\newtheorem{assumption}{Assumption}[]
\newtheorem{lemma}{Lemma}[]
\newtheorem{corollary}{Corollary}[]
\newtheorem{theorem}{Theorem}[]
\theoremstyle{definition}
\newtheorem{remark}{Remark}[]
\newtheorem{definition}{Definition}[]

\numberwithin{equation}{section}
\numberwithin{theorem}{section}
\numberwithin{lemma}{section}
\numberwithin{assumption}{section}
\numberwithin{corollary}{section}
\numberwithin{remark}{section}
\numberwithin{definition}{section}

\title{Fractional Backward Stochastic Partial Differential Equations with Applications to Stochastic Optimal Control of \\Partially Observed Systems driven by Lévy Processes}

\author{
Yuyang Ye\thanks{Department of Finance and Control Sciences, School of Mathematical Sciences, Fudan University, Shanghai 200433, China. {\small\it E-mail:} {\small\tt 22110180052@m.fudan.edu.cn}.}
\and
Yunzhang Li\thanks{Corresponding Author. Research Institute of Intelligent Complex Systems, Fudan University, Shanghai 200433, China. {\small\it E-mail:} {\small\tt li\_yunzhang@fudan.edu.cn}. Research supported by the National Natural Science Foundation of China (No.~12301566), and by the Science and Technology Commission of Shanghai Municipality (No.~23JC1400300), and by the Chenguang Program of Shanghai Education Development Foundation and Shanghai Municipal Education Commission (No.~22CGA01).}
\and
Shanjian Tang\thanks{Department of Finance and Control Sciences, School of Mathematical Sciences, Fudan University, Shanghai 200433, China. {\small\it E-mail:} {\small\tt sjtang@fudan.edu.cn}. Research supported by the National Key R\&D Program of China (Grant No. 2018YFA0703900) and the National Science Foundation of China (Grant No. 11631004).}
}
\date{}

\allowdisplaybreaks[4]

\begin{document}

\maketitle
\begin{abstract}
In this paper, we study the Cauchy problem for backward stochastic partial differential equations (BSPDEs) involving fractional Laplacian operator. Firstly, by employing the martingale representation theorem and the fractional heat kernel, we construct an explicit form of the solution for fractional BSPDEs with space invariant coefficients, thereby demonstrating the existence and uniqueness of strong solution. Then utilizing the freezing coefficients method as well as the continuation method, we establish Hölder estimates and well-posedness for general fractional BSPDEs with coefficients dependent on space-time variables. As an application, we use the fractional adjoint BSPDEs to investigate stochastic optimal control of the partially observed systems driven by $\alpha$-stable Lévy processes.

\bigskip


\medskip

\end{abstract}

\section{Introduction}
In this paper, we consider the Cauchy problem for a class of nonlocal backward stochastic partial differential equations (BSPDEs for short):
\begin{align}\label{linear-case 2}
\left\{
\begin{aligned}
  -\D u(t,x)=&\Big[-a(-\Delta)^{\frac{\alpha}{2}}u+b{ u_x}+cu+f +\sigma v\Big](t,x) \D t-v(t,x)\D W_t,\quad  (t,x)\in [0,T)\times \mathbb{R},\\
    u(T,x)=&\,\,g(x),\qquad\qquad\qquad\qquad\qquad\qquad\qquad\quad\qquad\qquad\qquad\qquad\qquad x\in\R,  
\end{aligned}
\right.
\end{align}
where $-(-\Delta)^{\frac{\alpha}{2}}$ is the fractional Laplacian operator with order $\alpha\in(1,2]$. The standard one-dimensional Brownian motion $W=\{W_t:t\in[0, T]\}$ is defined on a completed probability space $(\Omega, \mathbb{P},\mathcal{F})$. Let $\mathbb{F}:=\{\mathcal{F}_t\}_{t\geq 0}$ be the natural filtration generated by $W$, augmented by all the $\mathbb{P}$-null sets in $\mathcal{F}$. The leading coefficient $a:[0, T]\times \R\rightarrow \R^+$ is a deterministic positive function of time-space variable $(t,x)$, and the coefficient $(b,c,f,\sigma):\Omega \times [0, T]\times \R \rightarrow \R^4$ as well as the terminal condition $g:\Omega \times \R \rightarrow \R$ are real-valued random functions. Our paper aims to study the solution $(u,v)$ to fractional BSPDE~\eqref{linear-case 2} in suitable Hölder space and apply the regularity results to investigate stochastic optimal control of the partially observed systems driven by $\alpha$-stable Lévy processes.

BSPDEs, an infinite dimensional extension of backward stochastic differential equations (BSDEs), play a crucial role in various applications in probability theory and stochastic optimal control theory. For example, in the context of partially observed optimal stochastic control problems, a linear BSPDE emerges as the adjoint equation of the Zakai equation, which is pivotal for establishing the maximum principle (see e.g. Zhou~\cite{Zhou1993}).
Furthermore, solving forward-backward stochastic differential equations (FBSDEs) with random coefficients is intimately linked to solving quasi-linear BSPDEs, which serve as the foundation for the stochastic Feynman–Kac formula (see e.g. Ma and Yong~\cite{Ma1997}). Another significant application is that the stochastic Hamilton-Jacobi-Bellman (HJB) equation is a kind of nonlinear BSPDE, which is derived from the dynamic programming principle for controlled non-Markovian SDEs (see e.g. Peng~\cite{Peng1992}).

The theory of BSPDEs with regular Laplacian operators in Sobolev spaces has been well-developed, as documented in, for example, \cite{Ma1997, Ma1999, Hu2002, Dukai2012}. Compared to Sobolev solutions, there are relatively few studies on the Hölder solutions to BSPDEs. The first Hölder regularity result was established by Tang and Wei~\cite{tang2016cauchy}, but only for the BSPDEs involving regular Laplacian, i.e., $\alpha = 2$. To our knowledge, our paper is the first investigation into BSPDEs involving fractional Laplacian, i.e., $\alpha\in(1,2]$. Despite the nonlocal nature of fractional operator $-(-\Delta)^{\frac{\alpha}{2}}$, we consistently achieve the robust local Hölder regularities and thoroughly corroborate the results for $\alpha=2$ as detailed in~\cite{tang2016cauchy}.

In last decades, there has been increasing attentions in stochastic partial differential equations (SPDEs) that incorporate the fractional Laplacian. Kim and Kim~\cite{Kim2012} studied the $L_p$ regularity for SPDEs with fractional Laplacians. Subsequently, the $L_p$ theory for fractional SPDEs has been explored by several papers (see e.g., \cite{Kim2013am, xie2014fractal, gyongy2021, Choi2024}). On the other hand, various works have investigated solutions in Hölder spaces. Debbi and Dozzi~\cite{debbi2005solutions} examined the existence, uniqueness, and regularity of trajectories of mild solutions for a class of nonlinear fractional SPDEs in Hölder spaces. Mikulevicius and Pragarauskas~\cite{Mikulevicius2009} discussed the Hölder solutions for the integral-differential Zakai equation, and recent studies about the Hölder regularity for SPDEs can be found in, for example, \cite{Niu2010, Abouagwa2019, Tian2019}. However, the BSPDE~\eqref{linear-case 2} that we concern in this paper differs significantly from the forward SPDE, since the backward SPDE introduces an additional unknown variable $v$, which generally has a lower regularity and appears as the integrand of the It\^o's integral. Therefore, working in an appropriate Hölder space is essential for our study. Motivated by the setting in~\cite{Mikulevicius2009} and~\cite{tang2016cauchy}, we regard the solution $(u,v)$ as a functional on $\R$ with values in $\mathcal{S}_{\mathbb{F}}^2(0,T) \times \mathcal{L}_{\mathbb{F}}^2(0,T)$ and consider the Hölder regularity of $ x\mapsto(u(\cdot,x),v(\cdot,x))$.

Due to the nonlocal nature of fractional operator $-(-\Delta)^{\frac{\alpha}{2}}$, most classical local analysis methods fail in this context. To overcome this inherent challenge, we first employ the fractional heat kernel function and the martingale representation theorem, so that we can construct an explicit form of solution for fractional BSPDEs when the coefficient $(a,\sigma)$ does not depend on spatial variable $x$ and $(b,c)\equiv 0$. This establishes the existence and uniqueness of solutions, and allows us to study their regularity properties. Furthermore, to investigate cases where the equation coefficients depend on spatial variable $x$, we leverage the method of freezing coefficients to successfully obtain the a priori Hölder estimates for the solutions. Based on the estimates, using continuation method, we ultimately establish well-posedness for general fractional BSPDEs. It is noteworthy that our paper focuses on solutions in Hölder spaces, thus our results are not included in the classical framework based on the Gelfand triple (see e.g. Liu and R\"ockner~\cite{Rockner2015}). We also emphasize that although our paper focuses on the BSPDEs with fractional Laplacian, our method can be naturally extended to study the BSPDEs with more general nonlocal operators.  

As an important application, fractional BSPDEs serve as the adjoint equations in maximum principle of partially observed optimal stochastic control problems with jumps. For this kind of control problem, we first mention that Hu and Øksendal~\cite{Hu2008} investigated the linear-quadratic (LQ) problem involving jumped diffusion and partial information control. Later, Øksendal and Sulem~\cite{Oksendal2009} explored the general maximum principle for optimal control with partial information in stochastic systems, characterized by FBSDEs driven by Levy processes. Ahmed and Charalambous~\cite{Ahmed2013} provided direct proof of the stochastic maximum principle for the jump-diffusion controlled processes under a relaxed control framework. Recently, Zheng and Shi~\cite{Zheng2023} investigated the progressive optimal control problem of FBSDEs with random jumps in a broader context, and Zhang and Chen~\cite{MR4717295,MR4758317} studied controlled systems driven by anomalous sub-diffusions, which are some jump processes that are widely used to model many natural systems. It is worth noting that another effective method for studying nonlinear partially observed control problems is the separation principle, which elevates the controlled state to an infinite-dimensional Zakai equation by introducing a reference probability $\mathbb{Q}$. This approach not only ensures the time-consistency for the control problem, but also provides the linear properties of controlled system. However, when the controlled system is driven by a Lévy process, it corresponds to fractional Zakai equations and fractional adjoint BSPDEs. In this framework, relevant research results are currently lacking. To address this gap, we apply the findings of this paper to characterize the regularity of the fractional adjoint BSPDEs, which helps us to apply spike variational techniques to successfully establish the global maximum principle for the partially observed optimal control problems with jumps.

The rest of the paper is structured as follows. In Section~\ref{secnotation}, we introduce notations and functional Hölder spaces, alongside fundamental concepts concerning the fractional Laplacian. In Section~\ref{secresults}, we present our main results concerning the existence, uniqueness, and regularity of solutions to fractional BSPDEs~\eqref{linear-case 2} in Hölder spaces. In Section~\ref{seccontrol}, we apply our findings to fractional adjoint BSPDEs arising in partially observed optimal control problems with jumps. Finally, in Section~\ref{secproof}, we give the proof of our main results. 

\section{Preliminaries}\label{secnotation}
\subsection{Notation and spaces}\label{section notation}
Suppose $X$ is a Banach space equipped with the norm $\|\cdot\|_X$. For $p\in[1,+\infty)$, it is well known that $L^p(\Omega,X)$ is a Banach space equipped with the norm
\begin{align*}
    \|\xi\|_{L^p(\Omega,X)}:=\E\big[ \|\xi\|_X^p\big]^{\frac{1}{p}}.
\end{align*}
If $X=\R$, we simply write the space as $(L^p,\|\cdot\|_{L^p})$. 

Let $\mathcal{L}_{\mathbb{F}}^p(0,T;X)$ be the Banach space of all $X$-valued $\mathbb{F}$-adapted stochastic process $f$ with the finite norm
\begin{align*}
    \|f\|_{\mathcal{L}^p(0,T;X)}:=\left\{\begin{aligned}
        &\E\Big[\int_0^T \|f(t)\|_{X}^p\D s\Big]^{\frac{1}{p}},\quad\quad 1\leq p<+\infty,\\[0.3cm]
        &\operatornamewithlimits{esssup}_{\omega,t} \|f(w,t) \|_{X},\quad\quad p=\infty.
    \end{aligned}\right.
\end{align*}
Let $\mathcal{S}_{\mathbb{F}}^p(0,T;X)$ be the set of $X$-valued $\mathbb{F}$-adapted strongly continuous process $f$ such that
\begin{align*}
    \|f\|_{\mathcal{S}^p(0,T;X)}:=\E\Big[\max_{t\in[0,T]} \|f(t)\|_X^p\Big]^{\frac{1}{p}}<\infty.
\end{align*}
If $X=\R$, we write the space $(\mathcal{L}_{\mathbb{F}}^p(0,T;\R), \|\cdot\|_{\mathcal{L}^p(0,T;\R)})$ as $(\mathcal{L}_{\mathbb{F}}^p(0,T), \|\cdot\|_{\mathcal{L}^p})$, and write the space $(\mathcal{S}_{\mathbb{F}}^p(0,T;\R), \|\cdot\|_{\mathcal{S}^p(0,T;\R)})$ as $(\mathcal{S}_{\mathbb{F}}^p(0,T), \|\cdot\|_{\mathcal{S}^p})$ for simplicity. 

Next we define the functional Hölder spaces. Suppose $m$ is a nonnegative integer and $\beta\in(0,1)$ is a constant. For $X=\mathcal{L}_{\mathbb{F}}^p(0,T)$ or $X=\mathcal{S}_{\mathbb{F}}^p(0,T)$, define $C^m(\R,X)$ to be the Banach space of all measurable functions $\phi:\Omega \times [0,T]\times \R\rightarrow \R$ that are $m$-times continuously differentiable with respect to $x$ for each $(\omega,t,x)\in\Omega\times[0,T]\times\R$ and satisfy that
\begin{align*}
    \|\phi\|_{m,X}:=\sum_{k=0}^m\left\|D^k \phi\right\|_{0,X}<\infty,\quad \text{with}\quad \|\phi\|_{0,X}:=\sup_{x}\|\phi(x)\|_{X},
\end{align*}
where $D^k \phi(\omega,t,x) = \partial^k_x \phi(\omega,t,x)$ denotes $k$-th order partial derivative with respect to $x$. Then we define the the Banach space $C^{m+\beta}({\R,X})$ (or $C^{m,\beta}({\R,X})$) to be the set of $\phi\in C^m({\R,X})$ such that 
\begin{align*}
    \left[\phi\right]_{m+\beta,X}:=\sum_{k=0}^m\left[D^k \phi\right]_{\beta,X}<\infty, \quad\text{where}\quad\left[\psi\right]_{\beta,X}:=\sup_{x\neq y}\frac{\|\psi(x)-\psi(y)\|_{X}}{|x-y|^\beta}.
\end{align*}
We can equip $C^{m,\beta}(\R,X)$ with the norm $\|\cdot\|_{m+\beta,X}:=\|\cdot\|_{m,X}+[\,\cdot\,]_{m+\beta,X}$. Similarly, we can define the functional Hölder spaces for different Banach spaces such as $X=L^2(\Omega,\R)$. By the definition, it naturally has the relationship
\begin{align*}
   C^{m+1}(\R,X)\subset  C^{m+\beta_1}(\R,X)\subset C^{m+\beta_2}(\R,X)\subset C^{m}(\R,X), \quad\quad \forall \,0<\beta_2<\beta_1<1.
\end{align*}
If $X=\R$, then the above Banach spaces  $C^{m+\beta}(\R,\R)$, $C^{m}(\R,\R)$ are just the classical Hölder spaces, and we simply write $C^{m+\beta}(\R), C^{m}(\R)$ with $\|\cdot\|_{m+\beta},[\,\cdot\,]_{m+\beta},\|\,\cdot\,\|_{m}$. 

Additionally, let us recall the fractional functional Soblev spaces on $\R$. Denote $L^2(\R)$ be set of function $u$ such that $\|u\|_{L^2(\R)}:=\Big(\int_{\R}|u|^2\D x\Big)^{\frac{1}{2}}<+\infty$. For $\gamma>0$, denote $H^{\gamma}$ to be fractional Soblev space with the norm $\|u\|_{H^\gamma}:=\|\left(1-\Delta\right)^\frac{\gamma}{2}u\|_{L^2(\R)}$
and define $\mathbb{H}^\gamma(T):=\mathcal{L}^2(\Omega\times [0,T],H^{\gamma})$.

\subsection{Fractional Laplacian operator}
We consider the function $u:\R\rightarrow\R$ in the Schwartz space $\mathcal{S}(\R)$, which is the set of all rapidly decreasing $C^{\infty}(\R)$ functions. For $\alpha \in (0,2]$, we use the definition in~\cite{Kwasnicki2017} for the fractional operator 
\begin{align*}
   (-\Delta)^{\frac{\alpha}{2}}u(x):= \mathcal{F}^{-1}\left(|\xi|^{\alpha}\mathcal{F}(u)(\xi)\right)(x),
\end{align*}
with
\begin{align*}
    \mathcal{F}(u)(\xi):=\int_{\R}u(x)e^{i\xi x}\D x\quad\quad \text{and} \quad\quad \mathcal{F}^{-1}(\hat{u})(x):=\frac{1}{2\pi}\int_{\R}\hat{u}(\xi)e^{-i\xi x}\D \xi
\end{align*}
being the Fourier transformation and Fourier inverse transformation, respectively. For $u\in  \mathcal{S}'(\R)$, the fractional Laplacian satisfies that
\begin{align}\label{eq202407101}
     \langle(-\Delta)^{\frac{\alpha}{2}}u,\phi
    \rangle_{S'(\R),S(\R)}= \langle u,(-\Delta)^{\frac{\alpha}{2}}\phi\rangle_{S'(\R),S(\R)},\quad\quad \forall \phi\in\mathcal{S}(\R).
\end{align}
Note that when $\alpha=2$, the definition of $\Delta^1$ is in accordance with classical Laplacian $\Delta$, which is a local operator.

Especially, when $\alpha\in (1,2)$ it can be shown that (see e.g., the proposition 2.4 in~\cite{silvestre2007regularity} for details)
\begin{align}\label{pointwise formula}
    (-\Delta)^\frac{\alpha}{2}u(x)= C_\alpha \, P.V.\int_{\R}\frac{u(x)-u(y)}{|x-y|^{1+\alpha}}\D y=C_\alpha\int_{\R}\frac{u(x)-u(y)-D u(x)\cdot(x-y) }{|x-y|^{1+\alpha}}\D y,
\end{align}
where $C_\alpha=\frac{2^\alpha \Gamma\left(\frac{1+\alpha}{2}\right)}{\pi^{\frac{1}{2}}\Gamma\left(-\frac{\alpha}{2}\right)}$ and $u\in C^{\alpha+{\varepsilon}}(\R)$ for some $\varepsilon>0$ with $\int_{\R}\frac{|u(x)|}{1+|x|^{1+\alpha}}\D x<\infty.$

\section{Main results}\label{secresults}
Throughout the paper, we consider the fractional Laplacian order $\alpha \in (1,2]$ and make the following boundedness assumptions for the coefficients.
\begin{assumption}\label{Baisc assumption}
The leading coefficient $a:[0,T]\times\R\rightarrow \R^+$ belongs to the space $C^{\beta}\left(\R,{L}^{\infty}(0,T)\right)$, and has a uniformly positive lower bound, i.e. there exists $m>0$ such that $a\geq m$. All the coefficients $a,b,c,\sigma$ are bounded, i.e. there exists $K>0$ such that
\begin{align*}
    |a|+|b|+|c|+|\sigma|\leq K.
\end{align*}
\end{assumption}

Furthermore, we assume that the coefficients of BSPDEs are in the corresponding Hölder space as follows.

\begin{assumption}\label{general assumption}
   The coefficient $(b, c,\sigma):\Omega \times [0,T] \times \R \rightarrow \R$ belong to the space $C^{\beta}\left(\R,\mathcal{L}_{\mathbb{F}}^{\infty}(0,T)\right)$ such that 
   $$
   \|a\|_{\beta,{L}^{\infty}}+\|b\|_{\beta,\mathcal{L}^{\infty}}+\|c\|_{\beta,\mathcal{L}^{\infty}}+\|\sigma\|_{\beta,\mathcal{L}^{\infty}}<\Lambda,
   $$
   where the constants $\beta\in (2-\alpha,1)$ and $\Lambda>0$.
\end{assumption} 

Our main result is the sharp Hölder estimate of BSPDEs as follows. 
\begin{theorem}\label{Theorem genral}
    Let Assumptions \ref{Baisc assumption} and \ref{general assumption} hold. If the non-homogeneous term $f:\Omega \times [0,T] \times \R \rightarrow \R$ belongs to the space $C^{\beta}\left(\mathbb{R},\mathcal{L}^2_\mathbb{F}(0, T)\right)$ and the terminal condition $g:\Omega \times \R \rightarrow \R$ belongs to the space $C^{\frac{\alpha}{2}+\beta}\left(\mathbb{R},L^2\right(\Omega))$, then there exists a unique classical solution $(u,v)$ to the fractional BSPDE~\eqref{linear-case 2}.
Moreover, we have
\begin{align}\nonumber
    \Vert u\Vert_{\alpha+\beta,\mathcal{L}^2}+ \Vert u\Vert_{\beta,\mathcal{S}^2}+\Vert v\Vert_{\beta,\mathcal{L}^2}\leq C\big(\Vert g\Vert_{\frac{\alpha}{2}+\beta,{L}^2}+\Vert f\Vert_{\beta,\mathcal{L}^2}\big),
\end{align}
where the constant $C=C(\alpha,\beta,\Lambda,T,m,K)$.
\end{theorem}

\begin{remark}
   The case of $\alpha=2$ is exactly Theorem 4.3 in \cite{tang2016cauchy}.
\end{remark}

\begin{remark}
    In fact, the fractional BSPDE is closely related to FBSDE driven by the jump process. Consider the jumped FBSDE as follows
    \begin{align}\label{FBSDE}
    \left\{
\begin{aligned}
    &\D X_t=b(t,X_t)\D t+a^\frac{1}{\alpha} (t,X_{t-})\,\D M_t, \\[0.1cm]
    &\D Y_t=-\left[c(t,X_t) Y_t+\sigma(t,X_t)Z_t+f(t,X_t)\right] 
   \D t+Z_t\D W_t +\int_{\R}K(t-,x) \tilde{N}(\D t,\D x)
    ,\\[0.1cm]
    &X_0=x, \quad Y_T=g(X_T).
\end{aligned}
    \right.
\end{align}
For the solution $(u,v)$ to BSPDE~\eqref{linear-case 2}, it yields from It{\^o}-Ventzell formula (see e.g. \cite{oksendal2007ito}) that
\begin{align*}
    \D u(t,X_t)=&p(t,X_t)\D t+q(t,X_t)\D W_t+b(t,X_t) D u(t,X_t)\D t\\[0.1cm]
    &+\int_{\R} \left(u(t,X_t+x\cdot a^\frac{1}{\alpha} (t,X_{t}))-u(t,X_t)-x\cdot D u(t,X_t)a^\frac{1}{\alpha} (t,X_{t})\mathbbm{1}_{|x|<1}\right)\D t\\
    &+\int_{\R} \left(u(t-,X_{t-}+xa^\frac{1}{\alpha} (t-,X_{t-})-u(t-,X_{t-})\right)\tilde{N}(\D t,\D x),
\end{align*}
where 
$$p(t,x)=a(t,x)(-\Delta)^{\frac{\alpha}{2}}v
   (t,x)-b(t,x)Du(t,x)-c(t,x)u(t,x)-\sigma(t,x)v(t,x).
   $$
Then we can verify that
\begin{align*}
  &(Y_t,Z_t,K(t,x))=(u(t,X_t),v(t,X_t),u(t,X_t+x\sigma(t,X_t))-u(t,X_t))
\end{align*}
is the solution of~\eqref{FBSDE}. In other words, FBSDE~\eqref{FBSDE} is a stochastic representation of fractional BSPDE~\eqref{linear-case 2}, and one can use the probabilistic approach to investigate the regularity of fractional BSPDE~\eqref{linear-case 2}.
\end{remark}

In particular, if the coefficients $a,b,c,f,\sigma,g$ are deterministic, then the BSPDE is degenerate to a PDE as follows,  i.e., $v\equiv 0$ and 
\begin{align}\label{linear-case 20240623}
\left\{
\begin{aligned}
  u_t(t,x)=&\left[a(-\Delta)^{\frac{\alpha}{2}}u+b{ u_x}+cu+f\right](t,x),\quad\quad  (t,x)\in [0,T)\times \mathbb{R},\\
    u(T,x)=&g(x),\qquad\qquad\qquad\qquad\qquad\qquad\qquad\qquad\quad x\in\R.   
\end{aligned}
\right.
\end{align}
Then by Theorem~\ref{Thm1}, we have the following corollary for the deterministic fractional heat equation.
\begin{corollary}\label{coro deterministic}
      Let Assumptions \ref{Baisc assumption} and \ref{general assumption} hold, and let the coefficients $a,b,c,f,\sigma,g$ be deterministic. If $f \in C^{\beta}\left(\mathbb{R},{L}^2(0, T)\right)$ and $g \in C^{\frac{\alpha}{2}+\beta}\left(\mathbb{R}\right)$, then there exists a unique classical solution $u$ to PDE~\eqref{linear-case 20240623}.
Moreover, the solution satisfies that
\begin{align}\nonumber
    \Vert u\Vert_{\alpha+\beta,L^2(0,T)}+ \Vert u\Vert_{\beta,C^0}\leq C\big(\| g\|_{\frac{\alpha}{2}+\beta}+\Vert f\Vert_{\beta,{L}^2(0,T)}\big),
\end{align}
where the constant $C=C(\alpha,\beta,\Lambda,T,m,K)$.
\end{corollary}

\section{Application: Partially observed optimal control problem with jumps}\label{seccontrol}
In this section, we apply our results of fractional BSPDEs to study the partially observed stochastic optimal control problems driven by $\alpha$-stable Lévy processes.

\subsection{Optimal control problem of fractional Zakai equation}

Let $(\Omega, \mathcal{F}, \mathbb{P})$ be a complete probability space, equipped with the standard one-dimensional Brownian motion $W$. We denote $N(\D t,\D x)$ to be the Possion measure with density measure $\nu(\D x)\D t=\frac{C_{\alpha}} {|x|^{1+\alpha}}\D x\D t$, where $C_\alpha$ is the constant defined in~\eqref{pointwise formula}. Let $\tilde{N}(\D t,\D x)= {N}(\D t,\D x) - \nu(\D x)\D t$. Then the $\alpha$-stable Lévy process $M$ is defined as
\begin{align}
    M_t:=\int_0^t\int_{|x|<1}x\tilde{N}(\D s,\D x)+\int_0^t\int_{|x|\geq 1}x{N}(\D s,\D x).\label{eq 821}
\end{align} 

In the partially observed problems, the admissible control $u$ is assumed to be adapted to the natural filtration $\mathbb{F}^Y:=(\mathcal{F}_t^Y)$ generated by observation $Y$, i.e. $\mathcal{F}_t^Y:=\sigma\left\{Y_s:s\leq t\right\}$. For a Borel set $U\subset \R$, let $U_{ad}\subset\mathcal{L}_{\mathbb{F}^Y}(0,T;U)$ denote all the admissible controls, which are stochastic processes taking values in $U$. For each $u\in U_{ad} $, the state process $X^u$ and the observation process $Y$ in the controlled system are governed by
\begin{align*}
\left\{
\begin{aligned}
     &\D X^u_t=k(t,X^u_t,u_t)\D t+\mu(t)\D M_t,\quad\quad X^u_0\sim \pi_0,\\
    & \D Y_t=h(t,X_t^u)\D t+\D W_t,\quad\quad Y_0=0,
\end{aligned}
\right.
\end{align*}
where the coefficient functions $k: \Omega\times [0, T]\times \R \times U \rightarrow \R$, $\mu: [0, T] \rightarrow \R$, $h:\Omega\times[0, T] \times \R\rightarrow \R$ and the initial probability measure $\pi_0$ are given. The cost functional is 
$$
J(u):=\E^{\mathbb{P}}\left[\int_0^T f(s,X^u_s,u_s)\mathrm{d} s+g(X^u_T)\right],
$$
where running cost $f:\Omega\times [0, T]\times \R \times U \rightarrow \R$ and terminal cost $g:\Omega\times \R\rightarrow\R$. We aim to find the optimal control $\bar{u}\in U_{ad} $ such that
\begin{align*}
J(\bar{u}) = \inf_{u\in\mathcal{U}_{ad}}J(u).
\end{align*}

The original problem in the above strong setting is ill-posed and time-inconsistent (see e.g.~\cite{Bensoussan1992,Tang1998}). The classical way to solve this issue is the so-called reference probability approach, in which the new probability $\mathbb{Q}^*$ is introduced as
$$
\frac{\mathrm{d}\mathbb{Q}^*}{\mathrm{d}\mathbb{P}}=(L^u_T)^{-1}\quad\text{with}\quad L^u_t=\exp\left(\int_0^th(s,X^u_s)\mathrm{d}Y_s-\frac12\int_0^t|h(s,X^u_s)|^2\mathrm{d}s\right).
$$
By Girsanov's theorem, $Y$ is a standard Brownian motion under the reference probability $\mathbb{Q}^*$. So the cost functional in the original partially observed problem is equivalent to 
\begin{align}\label{eq 812}
J(u) = \E^{\mathbb{Q}^*}\left[\int_0^T L_s^u f(s,X^u_s,u_s)\mathrm{d} s + L_T^u g(X^u_T)\right],
\end{align}
which makes the control problem to be well-posed.

Next we introduce the unnormalized conditional density of state $X^u$ to make the control problem to be time-consistent. Given an admissible control $u\in\mathcal{U}_{ad}$, denote $p^u$ to be the solution of the fractional Zakai equation
\begin{align}\label{eq829}
    \mathrm{d} p^u(t,x)=\left[-a(t)(-\Delta)^{\frac{\alpha}{2}}p^u-D(k^u p^u)\right](t,x)\mathrm{d} t+ (h p^u)(t,x)\mathrm{d} Y_t,\quad\quad p^u(0,x)=p_0(x),
\end{align}
where $a(t)=\left|\mu(t)\right|^\alpha$ and $k^u(t,x)=k(t,x,u_t)$. It is well-known that $p^u$ is the unnormalized conditional density of state $X^u$ under the reference probability $\mathbb{Q}^*$ (see e.g. Theorem 3.6 of \cite{Ceci2014}),
$$
\pi_t^u(\phi):=\E^{\mathbb{Q}^*} \left[ L^u_t \phi(X_t^u)\big|\mathcal{F}^Y_t\right]=\int_\R  \phi(x)p^u(t,x)\mathrm{d}x,\quad\quad \forall\, \phi \in C_b(\R).
$$ 
With the unnormalized conditional density $p^u$, the cost functional in~\eqref{eq 812} can be written as
\begin{align}\label{eq 824}
    J(u)=\E^{\mathbb{Q}^*}\left[\int_0^T \int_{\R}f(s,x,u_s)p^u(s,x)\mathrm{d} x\mathrm{d} s+\int_{\R}g(x)p^u(T,x)\mathrm{d} x\right].
\end{align}

Thus it is equivalent to study the stochastic control problem with the cost functional~\eqref{eq 824} and the infinite-dimensional states~\eqref{eq829}, which is well-posed and time consistent. Apart from ensuring well-posedness and time-consistency, elevating the dimension of state equation to infinity introduces additional benefits, notably the transformation of the state equation~\eqref{eq829} into a linear and homogeneous form.

Next we show that, under this infinite-dimensional setting the adjoint equation in the maximum principle is exactly the fractional BSPDE that we concern in this paper. When driving the maximum principle for the partially observed problem with the jumped system, we will work on the reference probability space $(\Omega,\Q^*,\mathcal{F})$. For convenience, we write the expectation as $\E$ instead of $\E^{\Q^*}$. And we make the following technical assumption.
\begin{assumption}\label{density} Suppose that
\begin{enumerate}[(i)]
    \item the initial distribution $\pi_0$ has a density with respect to the Lebesgue measure 
    $$
    p_0\in C^{\frac{\alpha}{2}+\beta}(\R,L^2(\Omega))\cap L^2(\Omega,H^{\frac{\alpha}{2}}(\R)); 
    $$
   \item the leading coefficient $a: [0,T] \rightarrow\R^+$ is non-degenerate and bounded, i.e. there exists a constant $K_0>0$ such that
     $$
     \frac{1}{K_0}\leq a\leq K_0;
     $$
   \item  the control domain $U$ is a Borel subset of $\R$, and the admissible control set is
   \begin{align*}
      {U}_{ad}:=\left\{u:[0,T]\times \Omega\rightarrow U \,\bigg|\, u \text{ is } \text{$\mathbb{F}^Y$-adapted and} \operatornamewithlimits{esssup}_{t\in[0,T]}\E[|u|^2]<\infty\right\};
   \end{align*}
    \item\label{assumption Hölder} for all $v\in U$, the drift coefficient $k(\cdot,v)$ is in $C^{1+\beta}(\R,\mathcal{L}^\infty)\cap \mathcal{L}_{\mathbb{F}}^2((0,T);H^1(\R))$ and the observation coefficient $h$ is in  $C^{\frac{\alpha}{2}+\beta}(\R,\mathcal{L}^\infty)$; moreover, there is a constant $K_1>0$ such that
   $$
\|k(\cdot,v)\|_{1+\beta,\mathcal{L}^\infty}+\|h\|_{\frac{\alpha}{2}+\beta,\mathcal{L}^\infty}\leq K_1;
   $$
   \item \label{assumption growth}
    the coefficient $k$ is continuously differentiable with respect to $x$ and $k,k_x,f$ is continuous with respect to $u$; moreover, there is a constant $K_2>0$ such that
   $$|f|+|g|<K_2;$$ 
   \item  \label{assumption SPDE}   for each fixed $v\in U$, the running cost $f(\cdot,v)$ is in  $C^{\beta}(\R,\mathcal{L}_{\mathbb{F}}^2(0,T))\cap \mathcal{L}_{\mathbb{F}}^2((0,T);L^2(\R))$ and the terminal cost $g$ is in $C^{\frac{\alpha}{2}+\beta}(\R,L^2(\Omega))\cap L^2(\Omega,L^2(\R))$.
\end{enumerate}

\end{assumption}




\subsection{Fractional adjoint equations}
We first define the weak solution of the fractional state equation.
\begin{definition}
    Given $u\in\mathcal{U}_{ad}$, we say $p^u\in\mathbb{H}^\alpha(T)$ is a weak solution of Zakai equation~\eqref{eq829}, if for all $\phi\in C^\infty_0(\R)$, it holds that
\begin{align*}
   \langle p^u(t,\cdot), \phi\rangle_{L^2}=&\langle p_0, \phi\rangle_{L^2}+\int_0^t\left[- a(s) \langle p^u(s,\cdot),(-\Delta)^{\frac{\alpha}{2}}\phi\rangle_{L^2}+(k(s,\cdot,u_s)p^u(s,\cdot),D \phi\rangle_{L^2}\right]\mathrm{d} s\\
&+\int_0^t\langle h(s,\cdot)p^u(s,\cdot),\phi\rangle_{L^2}\mathrm{d} Y_s,
\end{align*}
for all $t\in[0,T]$, $\mathbb{Q}^*$-a.e.
\end{definition} 

In view of Theorem 2.15 in \cite{Kim2012}, we get the following $L^p$ estimate for the solution to fractional Zakai equation.
\begin{lemma}\label{lemma240713}
   Let Assumption \ref{density} hold. For each $u\in \mathcal{U}_{ad}$, the Zakai equation~\eqref{eq829} has a unique weak solution $p^u\in \mathbb{H}^\alpha$ such that
    \begin{align*}
        \E\int_0^T\|p^u(s)\|^2_{H^\alpha}\mathrm{d}s\leq C \, \E\|p_0\|^2_{H^\frac{\alpha}{2}},
    \end{align*}
    where $C=C(K_0,K_1,T,\alpha)>0$ is a constant.
\end{lemma}

The next lemma is concerned with Hölder continuity of solution to the fractional Zakai equation.
\begin{lemma}
    Let Assumption \ref{density} hold. For each $u\in \mathcal{U}_{ad}$, there exists a unique strong solution $p^u\in C^{{\alpha}+\beta}(\R,\mathcal{L}^2$) to the fractional Zakai equation~\eqref{eq829} such that
    $$
    \|p^u\|_{\alpha+\beta,\mathcal{L}^2}\leq C \,  \|p_0\|_{\frac{\alpha}{2}+\beta,L^2},
    $$
    where $C=C(\alpha,\beta,T,K_0,K_2)$ is a constant.
\end{lemma}
\begin{proof}
    Given $u\in\mathcal{U}_{ad}$, the uniqueness of the solution can be obtained from Lemma~\ref{lemma240713}. 
    
    Now we consider the unique solution $p^{u,1}$ of the homogeneous PDE as follows
    \begin{align*}
        \mathrm{d} p^{u,1}=\left[-a(t)(-\Delta)^{\frac{\alpha}{2}}p^{u,1} -D(k^u p^{u,1})\right]\mathrm{d} t , \quad\quad p^{u,1}(0,x)=p_0(x). 
    \end{align*}
Then by Corollary \ref{coro deterministic}, we have     \begin{align}\nonumber
    \Vert p^{u,1} (\omega,\cdot) \Vert_{\alpha+\beta, L^2(0,T)}\leq C\,\| p_0(\omega,\cdot) \|_{\frac{\alpha}{2}+\beta}, \quad \mathbb{Q}^*-a.e.
\end{align}

Next we consider $p^{u,2}$, which is the unique solution to the following non-homogeneous SPDE
    \begin{align}\label{eq 1294}
   \mathrm{d} p^{u,2}=\left[-a(t)(-\Delta)^{\frac{\alpha}{2}}p^{u,2}-D(k^u p^{u,2}) \right]\mathrm{d} t +\left(h p^{u,2} +g^u\right)\mathrm{d}Y_t, \quad\quad p^{u,2}(0,x)=0,
    \end{align}
where the coefficient $g^u(t,x)=h(t,x)p^{u,1}(t,x)$ satisfies that
$$
\|g^u\|_{\frac{\alpha}{2}+\beta,\mathcal{L}^2}\leq C \, \|p^{u,1}\|_{\alpha+\beta,\mathcal{L}^2}\leq C \, \|p_0\|_{\frac{\alpha}{2}+\beta,{L}^2}.
$$
In view of Theorem 5 in \cite{Mikulevicius2009}, the unique solution of~\eqref{eq 1294} satisfies
\begin{align*}
    \|p^{u,2}\|_{\alpha+\beta,\mathcal{L}^2}\leq C \|g^u\|_{\frac{\alpha}{2}+\beta,\mathcal{L}^2}\leq C \, \|p_0\|_{\frac{\alpha}{2}+\beta,{L}^2}.
\end{align*}Then $p^u=p^{u,1}+p^{u,2}$ is the unique solution of~\eqref{eq829} such that
$$
\|p^{u}\|_{\alpha+\beta,\mathcal{L}^2}\leq \|p^{u,1}\|_{\alpha+\beta,\mathcal{L}^2}+\|p^{u,2}\|_{\alpha+\beta,\mathcal{L}^2}\leq C \, \|p_0\|_{\frac{\alpha}{2}+\beta,{L}^2}.
$$
\end{proof}

To introduce the adjoint equations for the control problem, for each $t\in[0,T]$ and $u\in\mathcal{U}_{ad}$, we define the linear operator $\mathcal{L}_t(u)$ and its dual operator $\mathcal{L}^*_t(u)$ as
    \begin{align}\label{eq404823}
        \left\{\begin{aligned}
           & \mathcal{L}_t({u})\phi(t,\cdot):=-a(t)(-\Delta)^{\frac{\alpha}{2}}\phi(t,\cdot)-D(k(t,\cdot,{u}_t)\phi(t,\cdot)),  \\
           & \mathcal{L}^*_t({u})\phi(t,\cdot):=-a(t)(-\Delta)^{\frac{\alpha}{2}}\phi(t,\cdot)+Dk(t,\cdot,{u}_t)\cdot\phi(t,\cdot).
        \end{aligned}\right.
    \end{align}

In summary, we get the well-posedness and regularity for the fractional adjoint equations as follows.
\begin{lemma}\label{lemma2407161}
    Let Assumption \ref{density} hold. For each admissible control $u\in\mathcal{U}_{ad}$, the forward-backward stochastic partial differential equation (FBSPDE) 
    \begin{align}\label{eq 861}
    \left\{
    \begin{aligned}
       & \mathrm{d} p^{{u}}(t,x)=\mathcal{L}_t({u})p^{{u}}(t,x)\mathrm{d} t+h(t,x)p^{{u}}(t,x)\mathrm{d} Y_t, \quad\quad p^{{u}}(0,x)=p_0(x),\\
        & \mathrm{d} q^u(t,x)=\left[-\mathcal{L}^*_t({u}) q^u-f(\cdot,{u}_t)-hl^u\right](t,x)\mathrm{d} t+l^u(t,x)\mathrm{d} Y_t,\quad\quad q^u(T,x)=g(x),
    \end{aligned}\right.
    \end{align}
    has the unique solution $(p^{{u}},q^u,l^u)$ such that
    \begin{align*}
    \left\{
    \begin{aligned}
    &p^{{u}} \in C^{{\alpha}+\beta}(\R,\mathcal{L}^2(0,T))\cap \mathcal{L}^2((0,T),H^{\alpha}(\R)),\\
    &q^u \in C^{{\alpha}+\beta}(\R,\mathcal{L}^2(0,T))\cap \mathcal{L}^2((0,T),H^\frac{\alpha}{2}(\R)),\\
    &l^u \in C^{\beta}(\R,\mathcal{L}^2(0,T))\cap \mathcal{L}^2((0,T),L^2(\R)).
    \end{aligned}\right.
    \end{align*}
\end{lemma}

\begin{proof}
    By Theorem \ref{Theorem genral}, there is an unique solution $({q}^u,{l}^u)\in C^{\alpha+\beta}(\R,\mathcal{L}^2(0,T))\times  C^\beta(\R,{\mathcal{L}^2}(0,T))$ such that
\begin{align*}
    \|{q}^u\|_{\alpha+\beta,\mathcal{L}^2}+\|{l}^u\|_{\beta,\mathcal{L}^2}\leq C\left(\|f^{{u}}\|_{\beta,\mathcal{L}^2}+\|g\|_{\frac{\alpha}{2}+\beta,{L}^2}\right).
\end{align*}
By the regular energy estimate of BSPDE, we obtain
\begin{align*}
    \E\int_0^T \|{q}^u\|^2_{H^{\frac{\alpha}{2}}}\mathrm{d}t
    +\E \sup_{t\in[0,T]} \|{q}^u\|^2_{L^2}
    + \E\int_0^T \|{l}^u\|^2_{L^2}\mathrm{d}t
    \leq C\left(\E\int_0^T \|f^{{u}}\|^2_{L^2}\mathrm{d}t+\|g\|^2_{L^2}\right).
\end{align*}
\end{proof}

\subsection{Maximum principle}
Now we state the main result in this section.  
\begin{theorem}\label{240626thm1}
    Let Assumption \ref{density} hold. If $\bar{u}$ is the optimal control of the problem~\eqref{eq829}-\eqref{eq 824}, then the optimal triple $(\bar{u},p^{\bar{u}},q^{\bar{u}})$ satisfies that $a.e.\, t\in[0,T],\,\mathbb{Q}^*-a.s.$,
   $$
    H(t,v,\bar{p}(t),\bar{q}(t))\geq H(t,\bar{u}_t,\bar{p}(t),\bar{q}(t)), \quad\quad \forall\,v\in U,
   $$
    where the Hamiltonian $H:[0,T]\times U\times H^1(\R)\times L^2(\R)\rightarrow\R$ is defined as
\begin{align*}
    H(t,v,p,q):=\langle f(t,\cdot,v), p\rangle_{L^2}-\langle D(k(t,\cdot,v)p),q \rangle_{L^2}.
\end{align*}
\end{theorem}

\begin{proof}
    Step 1: Denote the optimal control as $\bar{u}\in\mathcal{U}_{ad}$. For simplicity, we write the corresponding state as $\bar{p}$ instead of $p^{\bar{u}}$. From Lemma \ref{lemma2407161}, it follows that $\bar{p}\in \mathcal{L}^2((0,T),H^{\alpha}(\R))$, hence $\bar{p}(t)\in L^2(\Omega,H^{\alpha}(\R))$ for almost everywhere $t\in[0,T]$. Fix $\bar{t}\in[0,T)$ such that $\bar{p}(\bar{t}\,)\in L^2(\Omega,H^{\alpha}(\R))$ along with $U$-valued, $\mathcal{F}_{\bar{t}}$\,-adapted random variable $v$. For any $\varepsilon\in(0,T-\bar{t}\,)$, define $u^\epsilon$ by
    \begin{align*}
         u^\varepsilon=\left\{
         \begin{aligned}
        &v,\quad t\in[\bar{t},\bar{t}+\varepsilon],\\
        &\bar{u},\quad t\in [0,T]\backslash [\bar{t},\bar{t}+\varepsilon].
\end{aligned}    \right.    
    \end{align*} Define
    $\delta p^\varepsilon:=p^{u^\varepsilon}-\bar{p}$, then $\delta p^\epsilon(t)= 0$ for $t\in[0,\bar{t}\,]$. Recalling definition of operator $\mathcal{L}_t(u)$ in~\eqref{eq404823}, we have
    \begin{align}\label{248231}
        \D\delta p^\varepsilon =\left[\mathcal{L}_t(\bar{u})\delta p^\varepsilon - D\Big(\big(k(u^\epsilon)-k(\bar{u})\big) p^{u^{\varepsilon}}\Big) \right]
        \D t+h \delta p^{\varepsilon} \D Y_t, \quad t\in[\bar{t},\bar{t}+\varepsilon],
    \end{align}
    and    \begin{align*}
        \D\delta p^\varepsilon =\mathcal{L}_t(\bar{u})\delta p^\varepsilon
        \D t+h \delta p^{\varepsilon} \D Y_t, \quad t\in(\bar{t}+\varepsilon,T].
    \end{align*}
    
Next, we claim that 
\begin{equation}\label{eq2407161}
 \E \sup_{\bar{t}\leq t\leq \bar{t}+\varepsilon} \|\delta p^\varepsilon\|_{H^\frac{\alpha}{2}}^2 = O(\varepsilon).
\end{equation}
To see this, apply Theorem 2.6 and Theorem 2.15 in \cite{Kim2012} to equation~\eqref{248231}, then 
\begin{align*}
     \E\sup_{\bar{t}\leq t\leq \bar{t}+\varepsilon} \|\delta p^\varepsilon\|_{H^\frac{\alpha}{2}}^2\leq  C\E\int_{\bar{t}}^{\bar{t}+\varepsilon}\left\|D\Big(\big(k(u^\epsilon)-k(\bar{u})\big)\cdot p^{u^{\varepsilon}}\Big)\right\|^2_{L^2}\D t\leq C\varepsilon.
\end{align*}

Step 2: Now we derive the first-order condition of cost functional. Since $\bar{u}$ is optimal, we have
\begin{align*}
    0\leq J(u^\varepsilon)-J(\bar{u})=& \E
    \int_{\bar{t}}^{\bar{t}+\varepsilon}[\langle f(t,\bar{u}),\delta p^\varepsilon\rangle_{L^2}+\langle f(t,u^{\varepsilon})-f(t,\bar{u}),p^{u^\varepsilon}\rangle_{L^2}]\D t \\
    &+ \E
    \int_{\bar{t}+\varepsilon}^{T}[\langle f(t,\bar{u}),\delta p^\varepsilon\rangle_{L^2}\D t+\E\langle g, \delta p^{\varepsilon}(T)\rangle_{L^2}.
\end{align*}
 Considering the adjoint equation~\eqref{eq 861} with the optimal control $\bar{u}$, for simplicity, we write the corresponding solution as $\bar{q}$. By Lemma~\ref{lemma2407161}, we know that for almost everywhere $t\in[0,T]$, it holds $\langle \bar{q}(t),\delta p^\epsilon(t)\rangle_{L^2}<\infty$, $\mathbb{Q}^*-a.e.$. Using Ito's formula for $\langle\bar{q}, \delta p^\varepsilon \rangle_{L^2}$, we have
 \begin{align*}
     \E
    \int_{\bar{t}}^{T}\langle f(t,\bar{u}),\delta p^\varepsilon\rangle_{L^2}\D t+\E\langle g, \delta p^{\varepsilon}(T)\rangle_{L^2}
   = -\E\int_{\bar{t}}^{\bar{t}+\varepsilon}
  \left\langle \bar{q},D\Big(\big(k(u^\varepsilon)-k(\bar{u})\big)\cdot p^{u^{\varepsilon}}\Big)\right\rangle_{L^2} \D t .
 \end{align*}
 Hence,
 \begin{align}\label{eq561}
    0\leq \E\int_{\bar{t}}^{\bar{t}+\varepsilon}\langle f(t,u^{\varepsilon})-f(t,\bar{u}),\bar{p}\rangle_{L^2}- \left\langle \bar{q},D\Big(\left(k(u^\varepsilon)-k(\bar{u})\right)\cdot \bar{p}\Big)\right\rangle_{L^2}\D t +\delta J,
 \end{align}
 where
 $$
 \delta J=\E\int_{\bar{t}}^{\bar{t}+\varepsilon}\langle f(t,u^{\varepsilon})-f(t,\bar{u}),\delta p^{\epsilon}\rangle_{L^2}- \left\langle \bar{q},D\Big(\left(k(u^\varepsilon)-k(\bar{u})\right)\cdot \delta p^{\epsilon}\Big)\right\rangle_{L^2}\D t  .
 $$
 It follows from Young's inequality that
 \begin{align*}
     |\delta J|
     \leq &\varepsilon^{\frac{1}{2}}\E\int_{\bar{t}}^{\bar{t}+\varepsilon}\|f(t,u^\varepsilon)-f(t,\bar{u})\|^2_{L^2}\D t+ \varepsilon^{-\frac{1}{2}}\E\int_{\bar{t}}^{\bar{t}+\varepsilon}\|\delta p^{\varepsilon}\|_{L^2}\D t\\
     &+\varepsilon^{\frac{1}{2}}\E\int_{\bar{t}}^{\bar{t}+\varepsilon}\|\bar{q}\|_{H^{\frac{\alpha}{2}}}\D t+\varepsilon^{-\frac{1}{2}}\E\int_{\bar{t}}^{\bar{t}+\varepsilon}\left\|(I-\Delta)^{-\frac{\alpha}{4}}D\Big(\left(k(u^\varepsilon)-k(\bar{u})\right)\cdot \delta p^{\epsilon}\Big)\right\|_{L^2}\D t\\
     \leq &\varepsilon^{\frac{1}{2}}\E\int_{\bar{t}}^{\bar{t}+\varepsilon}\|f(t,u^\varepsilon)-f(t,\bar{u})\|^2_{L^2}\D t+\varepsilon^{\frac{1}{2}}\E\int_{\bar{t}}^{\bar{t}+\varepsilon}\|\bar{q}\|_{H^{\frac{\alpha}{2}}}\D t+\varepsilon^{-\frac{1}{2}}\E\int_{\bar{t}}^{\bar{t}+\varepsilon}\|\delta p^{\varepsilon}\|_{H^\frac{\alpha}{2}}\D t.
 \end{align*}
 According to Assumption \ref{density}~\eqref{assumption SPDE}, Lemma \ref{lemma2407161} and \eqref{eq2407161}, we have $\delta J=o(\varepsilon)$ for almost everywhere $\bar{t}\in[0,T]$. Dividing~\eqref{eq561} by $\varepsilon$ and letting $\varepsilon$ tend to 0, we obtain the desired results.
\end{proof}

\section{Proof of the main results}\label{secproof}
\subsection{Fractional BSPDE with space invariant coefficients $a(t)$ and $\sigma(t)$}
We first consider the model BSPDEs that the coefficient $(a,\sigma)$ is independent of space variable $x$ and $(b,c)\equiv 0$, so that we can get an explicit formulation for the solution. To be specific, we make the following assumption in this subsection.

\begin{assumption}\label{Simple assumption}
We assume that $a:[0, T] \rightarrow \R^+$ is a deterministic positive function, and $\sigma:\Omega \times [0, T] \rightarrow \R$ and $f:\Omega \times [0, T]\times \R \rightarrow \R$ as well as $g:\Omega \times \R \rightarrow \R$ are real-valued random functions, and $(b,c)\equiv 0$. 
\end{assumption}

Under the Assumption \ref{Simple assumption}, the BSPDE~\eqref{linear-case 2} can be written as 
\begin{align}\label{linear-case}
    \D u(t,x)=\left[a(t)(-\Delta)^{\frac{\alpha}{2}}u(t,x)-f(t,x)-\sigma(t) v(t,x)\right]\D t+v(t,x)\D W_t,\quad\quad u(T,x)=g(x).
\end{align}

\subsubsection{Fractional heat kernel and explicit formulation of solution}
Next, we introduce the heat kernel generated by fractional Laplacian
\begin{align}\label{eq 166}
   G_{t,s}(x)=\frac{1}{2\pi} \int_{\mathbb{R}}\exp\Big(-i\xi x-|\xi|^{\alpha}\int_s^t a(r)\D r\Big)\D \xi ,
\end{align}
where $x\in\R$ and $0\leq s <t \leq T$. We can check that $G_{t,s}(x)$ is the solution of the following equations
\begin{align}\label{240625eq}
   \left\{\begin{aligned}
      & \partial_t G_{t,s}(x)=-a(t)(-\Delta)^{\frac{\alpha}{2}}G_{t,s}(x),\\
     &  \partial_s G_{t,s}(x)=a(s)(-\Delta)^{\frac{\alpha}{2}}G_{t,s}(x).
   \end{aligned} \right.
\end{align}

For the deterministic PDE, i.e., the coefficients $a,g,f$ are deterministic and $v\equiv 0$, 
\begin{align}\label{linear-case 20240624}
    u_t(t,x)=a(t)(-\Delta)^{\frac{\alpha}{2}}u(t,x)-f(t,x),\quad\quad u(T,x)=g(x),
\end{align}
it is well known that the solution to PDE~\eqref{linear-case 20240624} can be formulated by the kernel~\eqref{eq 166} as follows
\begin{align*}
    u(t,x)=R^T_t g(x)+\int_t^T R_t^\tau f(\tau)(x)\D \tau,
\end{align*}
where semigroup operator $R_s^t$ is defined as 
\begin{align}
\left\{\begin{aligned}
     & R_s^t \phi(x) =\int_{\mathbb{R}} G_{t,s}(x-y)\phi(y)\D y,\\
        &  R_s^t \psi(t)(x) =\int_{\mathbb{R}} G_{t,s}(x-y)\psi(t,y)\D y.
\end{aligned}\right.
\end{align}

In order to give the explicit form of the solution $(u,v)$ to~\eqref{linear-case}, we define 
\begin{equation}\label{240626eq1}
 \widetilde{W}_t:=W_t-\int_0^t \sigma(s)\D s.   
\end{equation}
By Girsanov transformation, we get that $\widetilde{W}$ is a standard Brownian motion under probability space $(\Omega,\mathcal{F},\mathbb{F},\Q)$, where
$$
\frac{\D \Q}{\D \mathbb{P}}= \exp\Big(\int_0^T \sigma(s) \D W_s-\int_0^T\frac{1}{2}|\sigma(s)|^2\D s \Big).
$$
Next we set 
$$
p(t;x):=\E^{\Q}\left[g(x)|\mathcal{F}_t\right],\quad\quad Y(t;s,x):=\E^{\Q}\left[f(s,x)|\mathcal{F}_t\right].
$$
By Assumption \ref{Baisc assumption}, we know that for all  $x$ and $t$, it holds that $g(x)\in L^2(\Omega)$ and $f(t,x)\in L^2(\Omega)$. It yields from the martingale representation theorem that
\begin{align}
\left\{\begin{aligned}
  \label{eq 87}  p(t;x)&=g(x)+\int_t^T\sigma(s)q(s;x)\D s-\int_t^T q(s;x)\D W_s,\\
    Y(t;\tau,x)&=f(\tau,x)+\int_t^\tau \sigma(s)Z(s;\tau,x)\D  s-\int_t^{\tau} Z(s;\tau,x)\D W_s,
    \end{aligned}\right.
\end{align}
where $x\in \R$ and $0\leq t\leq \tau \leq T$.

Now we will use $p, q, Y, Z$ and the semigroup operator generated by the fractional heat kernel to represent $(u,v)$ so that the Hölder estimates can be derived with the help of explicit structure. The improved regularity is a consequence of the better properties of fractional heat kernel $G_{t,s}(x)$, whose estimates are given in the Appendix.
\begin{theorem}\label{Thm1}
Let Assumptions \ref{Baisc assumption} and \ref{Simple assumption} hold. There exists a unique classical solution $(u,v)$ to BSPDE~\eqref{linear-case} with the explicit form
    \begin{align}
    \left\{\begin{aligned}\label{eq108} 
       u(t,x)=R^T_t p(t)(x)+\int_t^T R_t^s Y(t;s)(x)\D s,\\
       v(t,x)=R^T_t q(t)(x)+\int_t^T R_t^s Z(t;s)(x) \D s. 
    \end{aligned}\right.
    \end{align}
Moreover, the solution $(u,v)\in \left( C_{\mathcal{S}^2}^{\beta}\cap C_{\mathcal{L}^2}^{\alpha+\beta} \right)\times C_{\mathcal{L}^2}^{\beta}$ satisfies that
\begin{align}\label{eq161}
    \Vert u\Vert_{\alpha+\beta,\mathcal{L}^2}+ \Vert u\Vert_{\beta,\mathcal{S}^2}+\Vert v\Vert_{\beta,\mathcal{L}^2}\leq C\big(\Vert g\Vert_{\frac{\alpha}{2}+\beta,{L}^2}+\Vert f\Vert_{\beta,\mathcal{L}^2}\big),
\end{align}
where $C=C(\alpha,\beta,K,T,m)$.
\end{theorem}

\begin{remark}\label{remark664}
Alternatively, using Fourier transformation and the result of linear BSDE, we can also directly get the explicit formulation of the solution $(u,v)$ to BSPDE~\eqref{linear-case}. To see this, we take Fourier transformation with respect to $x$ on both sides of equation~\eqref{linear-case},
    \begin{align}\label{eq2409021}
    \D \hat{u}(t,\xi)=\left[a(t)|\xi|^{\alpha}\hat{u}(t,\xi)-\hat{f}(t,\xi)-\sigma(t) \hat{v}(t,\xi)\right]\D t+\hat{v}(t,\xi)\D W_t,\quad\quad \hat{u}(T,\xi)=\hat{g}(\xi),
\end{align}
where 
$$
\hat{u}(t,\xi)=\mathcal{F}(u(t,\cdot))(\xi),\quad \hat{v}(t,\xi)=\mathcal{F}(v(t,\cdot))(\xi),\quad \hat{f}(t,\xi)=\mathcal{F}(u(t,\cdot))(\xi),\quad \hat{g}(\xi)=\mathcal{F}(g)(\xi).
$$
For any fixed $\xi\in \R$, equation~\eqref{eq2409021} is actually a linear BSDE. By the standard result of linear BSDE, we obtain that under some regularity condition, the unique solution to~\eqref{eq2409021} satisfies
\begin{align}\label{eq2409061}
    \hat{u}(t,\xi)=\E\left[\Gamma_{T,t}(\xi)\hat{g}(\xi)+\int_t^T\Gamma_{s,t}(\xi)\hat{f}(s,\xi)\D s \bigg| \mathcal{F}_t\right],
\end{align}
with 
$$
\Gamma_{s,t}(\xi)=\exp\left(\int_t^s \sigma(r)\D W_r+\int_t^s\big[-a(r)|\xi|^{\alpha}-\frac{1}{2}|\sigma(r)|^2\,\big]\D r\right),
$$
which easily leads to the existence and uniqueness of BSPDE within the Sobolev framework (see Appendix~\ref{Appendix A} for more details). However, The explicit formula~\eqref{eq2409061} seems to be still far from our purpose of addressing the well-posedness in Hölder space even when $a$ is deterministic. Thus, we prefer to derive the Hölder solutions via the fractional heat kernel in the article.
\end{remark}

\subsubsection{Existence and uniqueness of solution}
The proof of Theorem \ref{Thm1} is lengthy and we split it into several parts. First, we show that every solution of~\eqref{linear-case} has the explicit formulation~\eqref{eq108}. Thus, we immediately get the uniqueness of the solution of~\eqref{linear-case}. The proof is similar to Theorem 3.3 in \cite{tang2016cauchy} , so we leave it to the Appendix.
\begin{lemma}\label{uniqueness}
    Let Assumptions \ref{Baisc assumption} and \ref{Simple assumption} hold. If $(u,v)$ is a classical solution of~\eqref{linear-case}, then it can be formulated as~\eqref{eq108}.
\end{lemma}

In view of the classical results for the regularity of BSDEs, we can get the regularity of the stochastic flows $(p,q)$ and $(Y, Z)$ as follows.

\begin{lemma}\label{pqYZ regular}
Let Assumptions \ref{Baisc assumption} and \ref{general assumption} hold. We have
\begin{align*}
\Vert p\Vert_{\frac{\alpha}{2}+\beta,\mathcal{S}^2}+ \Vert q\Vert_{\frac{\alpha}{2}+\beta,\mathcal{L}^2}\leq C\Vert g\Vert_{\frac{\alpha}{2}+\beta,L^2},
\end{align*}
and
\begin{align*}
 &\sup_x \E\int_0^T\sup_{t\leq r}|Y(t;r,x)|^2\D r+\sup_{x\neq \bar{x}} \E\int_0^T\sup_{t\leq r}\frac{|Y(t;r,x)-Y(t;r,\bar{x})|^2}{|x-\bar{x}|^{2\beta}}\D r\\
            + &\sup_x \E\int_0^T\int_0^r|Z(t;r,x)|^2\D t\D r+\sup_{x\neq \bar{x}} \E\int_0^T \int_0^r \frac{|Z(t;r,x)-Z(t;r,\bar{x})|^2}{|x-\bar{x}|^{2\beta}}\D t\D r\leq C\Vert f\Vert_{\beta,\mathcal{L}^2}^2.
\end{align*}

\end{lemma}

With the regularity of $(p,q)$ and $(Y, Z)$, we can directly use the explicit formula~\eqref{eq108} to check the existence of BSPDEs~\eqref{linear-case} as follows.
\begin{lemma}\label{existence}
   Let Assumptions \ref{Baisc assumption} and \ref{Simple assumption} hold. Then the pair $(u,v)$ defined in~\eqref{eq108} is a solution to BSPDE~\eqref{linear-case}.
\end{lemma}
\begin{proof}
By the property~\eqref{240625eq} of fractional heat kernel and the definition~\eqref{240626eq1} of Brownian motion $\widetilde{W}$ under the probability $\mathbb{Q}$, we get for $s\in[0,T)$,
    \begin{align*}
     \D G_{T,s}(x-y)p(s;y)=a(s)(-\Delta)^{\frac{\alpha}{2}}G_{T,s}(x-y)p(s;y)\D s + G_{T,s}(x-y)q(s;y)\D \widetilde{W}_s.
    \end{align*}
For $r\in [t,T)$, integrating with respect to $s$ from $t$ to $r$ and integrating with respect to $y$ on $\mathbb{R}$, we get
   \begin{align}\label{eq124}
      R_r^T p(r)(x)= R_t^T p(t)(x)+\int_t^r a(s)(-\Delta)^{\frac{\alpha}{2}}R_s^Tp(s)(x)\D s + \int_t^r R_s^T q(s)(x)\D \widetilde{W}_s.
    \end{align}
    
Next, we let $r$ tends to $T$ and aim to show for all $t\in[0,T]$
 \begin{align}\label{eq 226}
      g(x)= R_t^T p(t)(x)+\int_t^T a(s)(-\Delta)^{\frac{\alpha}{2}}R_s^Tp(s)(x)\D s + \int_t^T R_s^T q(s)(x)\D \widetilde{W}_s.
 \end{align}
By the boundedness of $a$, it suffices to prove that equation~\eqref{eq 226} is continuous to the boundary, that is,
  \begin{align}
   \label{lim rt1}     &\lim_{r\rightarrow T} R^T_r p(r)(x)= g(x),\\
   \label{lim rt2}       & \lim_{r\rightarrow T} \int_r^T (-\Delta)^{\frac{\alpha}{2}}R_s^Tp(s)(x)\D s=0, \\
    \label{lim rt3}      &\lim_{r\rightarrow T} \int_r^T R_s^T q(s)(x)\D \widetilde{W}_s =0.
    \end{align}
for all $x\in\R$, $\mathbb{P}$-a.s.

First, we prove~\eqref{lim rt1}. Notice
\begin{align*}
  &\E\left[|R_r^Tp(r)(x)-g(x)|^2\right]=\E\left[\Big|\int_{\R} G_{T,r}(x-y)(p(r;y)-p(T;x)) \D y\Big|^2\right] \\
  \leq& \E\int_{\R}G_{T,r}(x-y)(p(r;y)-p(T;x))^2\D y = \int_{\R} G(z)\E\left[\left|p(r;x-A_{T,r}^{\frac{1}{\alpha}}z)-p(T;x)\right|^2\right]\D z\\
  \leq& 2\int_{\R} G(z)\E\left[\left|p(r;x-A_{T,r}^{\frac{1}{\alpha}}z)-p(r;x)\right|^2\right]\D z+2 \int_{\R} G(z)\E\left[\left|p(r;x)-p(T;x)\right|^2\right]\D z,
\end{align*}
where $G$ and  $A_{s,t}$ are the functions defined in~\eqref{defG} and~\eqref{defA} respectively, and the first inequality is due to Lemma~\ref{semi-group}~\eqref{int is 1} for the fractional heat kernel $G_{T,r}$, and the second equality comes from Lemma \ref{semi-group}~\eqref{variable change} when $k=0$. 
By Lemma \ref{pqYZ regular}, we have
\begin{align*}
   \lim_{r\rightarrow T} \E\left[\left|p(r;x-A_{T,r}^{\frac{1}{\alpha}}z)-p(r;x)\right|^2\right]\leq C \lim_{r\rightarrow T}A_{T,r}^{\frac{2\beta}{\alpha}}\,z^{2\beta }=0.
\end{align*}
It follows from the dominated convergence theorem that
\begin{align*}
  \lim_{r\rightarrow T}\int_{\R} G(z)\E\left[\left|p(r,x-A_{T,r}^{\frac{1}{\alpha}}\,z)-p(r,x)\right|^2\right]\D z =0.  
\end{align*}
Utilizing the continuity of $p(\cdot,x)$, we immediately obtain
\begin{align*}
   \lim_{r\rightarrow T} \int_{\R} G(z)\E\left[\left|p(r,x)-p(T,x)\right|^2\right]\D z=0.
\end{align*}
Therefore we complete the proof of limit~\eqref{lim rt1}.

To prove~\eqref{lim rt2}, by taking $\gamma=\frac{\alpha}{2}$, we have
\begin{align*}
    &\E\bigg[\Big|\int_r^T (-\Delta)^{\frac{\alpha}{2}}R_s^Tp(s)(x)\D s\Big|^2\bigg]\\
    =& \E\bigg[\Big|\int_r^T \int_{\R} (-\Delta)^{\frac{\alpha}{2}}G_{T,s}(x-y)(p(s;y)-p(s;x))\D y\D s\Big|^2\bigg]\\
    \leq & \E\bigg[\int_r^T \int_{\R} \left|(-\Delta)^{\frac{\alpha}{2}}G_{T,s}(x-y)\right||x-y|^{\gamma}\D y\D s\\
    &\quad\quad\quad\times\int_r^T \int_{\R}\left| (-\Delta)^{\frac{\alpha}{2}}G_{T,s}(x-y)\right||x-y|^{\gamma}\frac{|p(s;y)-p(s;x)|^2}{|x-y|^{2\gamma}}\D y\D s\bigg]\\
    \leq& \Big(\int_r^T\int_{\R} \Big|(-\Delta)^{\frac{\alpha}{2}}G_{T,s}(x-y)\Big||x-y|^{\gamma}\D y\D s\Big)^2\sup_{y}\E\left[\sup_{s\in[0,T]}\frac{|p(s;y)-p(s;x)|^2}{|x-y|^{2\gamma}}\right]\\
    \leq &\, C\,\|g\|^2_{\frac{\alpha}{2}+\beta}(T-r)^{\frac{2\gamma}{\alpha}},
\end{align*}
where the first equation comes from~\eqref{eq 175}, and the first inequality is just Cauchy's inequality, and the last inequality follows from Lemma \ref{pqYZ regular} as well as Lemma \ref{all G estimate}~\eqref{Lemma G 195} since $\gamma\in(0,\alpha)$. Thus,
\begin{align*}
    \lim_{r\rightarrow T}\E\left[\Big|\int_r^T (-\Delta)^{\frac{\alpha}{2}}R_s^Tp(s)(x)\D s\Big|^2\right] =0.
\end{align*}

Next, we prove~\eqref{lim rt3}. By BurkHölder-Davis-Gundy inequality (BDG inequality for short),
\begin{align*}
    &\E^\mathbb{Q} \left[\Big|\int_r^T R_s^T q(s)(x)\D \widetilde{W}_s\Big| \right] \\
    \leq& C \E^\mathbb{Q} \left[ \left( \int_r^T \left( \Big|\int_{\R} G_{T,s}(x-y) [q(s;y)-q(s;x)]\D y\Big|^2+|q(s;x)|^2 \right) \D s \right)^{\frac{1}{2}} \right]\\
    \leq& C \E^\mathbb{Q} \left[ \left( \int_r^T \int_{\R} G_{T,s}(x-y) |q(s;y)-q(s;x)|^2 \D y \D {s} \right)^\frac{1}{2} \right] + C\E^\mathbb{Q} \left[\left( \int_r^T|q(s;x)|^2\D s\right)^\frac{1}{2}\right] \\
    \leq& C \E^\mathbb{P} \left[ \int_r^T \int_{\R} G_{T,s}(x-y) |q(s;y)-q(s;x)|^2 \D y \D {s} \right]^\frac{1}{2} + C\E^\mathbb{P} \left[ \int_r^T|q(s;x)|^2\D s \right]^\frac{1}{2} \\
    \leq & C \left(\int_{\R} \sup_{s\in[0,T]}G_{T,s}(x-y)|x-y|^{\gamma}\D y \cdot\sup_{y}\E^\mathbb{P} \left[\int_r^T \frac{|q(s;y)-q(s;x)|^2}{|y-x|^{\gamma}}\D s \right] \right)^\frac{1}{2} \\
    & + C\E^\mathbb{P} \left[ \int_r^T|q(s;x)|^2\D s \right]^\frac{1}{2},
\end{align*}
where $0<\gamma<\alpha\wedge 2\beta$. By Lemma \ref{all G estimate}~\eqref{sup G estimate} and Lemma \ref{pqYZ regular}, we have
$$
\lim_{r\rightarrow T} \left(\int_{\R} \sup_{s\in[0,T]}G_{T,s}(x-y)|x-y|^{\gamma}\D y \cdot\sup_{y}\E^\mathbb{P} \left[\int_r^T \frac{|q(s;y)-q(s;x)|^2}{|y-x|^{\gamma}}\D s \right] \right) =0,
$$
and
$$
\lim_{r\rightarrow T} \E^\mathbb{P} \left[ \int_r^T|q(s;x)|^2\D s \right] =0,
$$
which imply~\eqref{lim rt3}, therefore~\eqref{eq 226} holds true.

Next, for $\epsilon>0$, denoting $r_t(\varepsilon)=(r-\varepsilon)\vee t$, we get
\begin{align*}
R_{r_t(\varepsilon)}^r Y(r_t(\varepsilon);r)(x)=R_t^rY(t;r)(x)+\int_t^{r_t(\varepsilon)}  a(u) (-\Delta)^{\frac{\alpha}{2}} R_u^rY(u;r)(x)\D u+\int_t^{r_t(\varepsilon)} R_u^rZ(u;r)(x)\D \widetilde{W}_u.
\end{align*}
In order to prove
\begin{align}\label{eq 234}
   \nonumber \int_t^T f(r,x)\D r =&\int_t^T R_t^rY(t;r)(x)\D r\\
    &+\int_t^T\int_t^{r}  a(u) (-\Delta)^{\frac{\alpha}{2}} R_u^rY(u;r)(x)\D u\D r+\int_t^T\int_t^{r} R_u^rZ(u;r)(x)\D \widetilde{W}_u\D r,
\end{align}
it suffices to show that
\begin{align}
   &\label{lim rt4} \lim_{\epsilon\rightarrow 0}\int_t^T R_{r_t(\varepsilon)}^r Y(r_t(\varepsilon);r)(x)\D r= \int_t^T f(r,x)\D r ,\\
   & \label{lim rt5}\int_t^T\int_{r_t(\varepsilon)}^{r}  (-\Delta)^{\frac{\alpha}{2}} R_u^rY(u;r)(x)\D u\D r=0,\\
   & \label{lim rt6}\int_t^T\int_{r_t(\varepsilon)}^{r} R_u^rZ(u;r)(x)\D \widetilde{W}_u\D r=0,
\end{align}
for all $x\in\R$, $\mathbb{P}$-a.s.

Notice that 
\begin{align*}
    &\E\left[  \Big|\int_t^T R_{r_t(\varepsilon)}^r Y(r_t(\varepsilon);r)(x)-f(r;x)\D r \Big|^2\right]\\
    \leq &C\int_{\R} G(z)\E\int_t^T|Y(r_t(\epsilon);r,x-A_{r_t(\epsilon),r}^{\frac{1}{\alpha}}z)-Y(r_t(\epsilon);r,x)|^2\D r \D z\\
    &+C\int_{\R} G(z)\E\int_t^T|Y(r_t(\epsilon);r,x)-Y(r;r,x)|^2\D r \D z,
\end{align*}
then with similar calculation to derive~\eqref{lim rt1}, we can obtain~\eqref{lim rt4}.

Similarly, with the help of Lemma \ref{all G estimate}~\eqref{sup G estimate} and ~\eqref{delta G sup int bound}, we have
\begin{align*}
    &\E\left[ \Big|\int_t^T\int_{r_t(\varepsilon)}^{r}  (-\Delta)^{\frac{\alpha}{2}} R_u^rY(u;r)(x)\D u\D r\Big|^2\right]\\
    \leq &C\Big(\int_{\R}\sup_{r\in[t,T]}\int_{r_t(\varepsilon)}^r\left|(-\Delta)^{\frac{\alpha}{2}} G_{r,u}(x-y)\right||x-y|^{\gamma}\D u\D y\Big)^2 \cdot \sup_{y}\E\int_t^T \sup_{u\leq r}\frac{|Y(u;r,y)-Y(u;r,x)|^2}{|x-y|^{2\gamma}}\D r,
\end{align*}
and
\begin{align*}
    &\E^{\mathbb{Q}}\left[ \Big|  \int_t^T\int_{r_t(\varepsilon)}^{r} R_u^rZ(u;r)(x)\D \widetilde{W}_u\D r \Big|\right] \leq \int_t^T \E^{\mathbb{Q}}\left[ \Big|\int_{r_t(\varepsilon)}^{r} R_u^rZ(u;r)(x)\D \widetilde{W}_u \Big|\right] \D r  \\
    \leq & C\Big(\int_{\R}\sup_{0\leq u<r\leq T} G_{r,u}(x-y)|x-y|^{\gamma}\D y\Big)^\frac{1}{2} \cdot \left( \int_t^T \sup_{y}\E \left[ \int_{r_t(\varepsilon)}^r \frac{|Z(u;r,y)-Z(u;r,x)|^2}{|y-x|^{\gamma}}\D u \right] \D r \right)^\frac{1}{2} \\
    &+C\,\E\left[ \int_t^T \int_{r_t(\varepsilon)}^r |Z(u;r,x)|^2\D u\D r \right]^\frac{1}{2}.
\end{align*}
Then by the similar argument in the proof of~\eqref{lim rt2} and~\eqref{lim rt3}, we can get~\eqref{lim rt5} and~\eqref{lim rt6}.

In the end, we complete the proof of the theorem by~\eqref {eq 226} and~\eqref{eq 234}.
\end{proof}

\subsubsection{Hölder estimates of solution}

The following Lemmas aim to give the Hölder estimates of the solution to~\eqref{linear-case}.
\begin{lemma}Let Assumptions \ref{Baisc assumption} and \ref{Simple assumption} hold. For the solution $(u,v)$ to BSPDE~\eqref{linear-case}, it holds that
    $$
    \Vert u\Vert_{\alpha+\beta,\mathcal{L}^2}\leq  C\left(\Vert g\Vert_{\frac{\alpha}{2}+\beta,L^2}+ \Vert f\Vert_{\beta,\mathcal{L}^2}\right),
    $$
    where $C=C(\alpha,\beta,K,T,m)$.
\end{lemma}
\begin{proof}
Firstly, we prove that
\begin{align}\label{eq 308}
    \Big\Vert\int_{\cdot}^T R_{\cdot}^s Y(\cdot;s)\D s\Big\Vert_{\alpha+\beta,\mathcal{L}^2}\leq C\Vert f\Vert_{\beta,\mathcal{L}^2}.
\end{align}

Note that $\alpha+\beta\in(2,3)$. For $\gamma=0,1$, we obtain
    \begin{align*}
        &\E\int_0^T \Big| D^\gamma \int_t^T R_t^sY(t;s)(x)\D s\Big|^2\D t\\
 \leq &  C \Big(\sup_{0\leq t \leq T} \int_t^T \int_{\R}\Big| D^\gamma G_{s,t}(x-y)\Big|\D y\D s \Big)^2\cdot\sup_{y}\E\int_0^T \sup_{t\leq s}|Y(t;s,y)|^2\D s \leq C 
 \Vert f\Vert^2_{\beta,\mathcal{L}^2},
    \end{align*} 
where the last inequality uses Lemma \ref{pqYZ regular} and Lemma \ref{all G estimate}~\eqref{Lemma G 200}. For $\gamma=2$, with the help of  Lemma \ref{pqYZ regular} and Lemma \ref{all G estimate}~\eqref{Lemma G 200}, we have
    \begin{align*}
        &\E\int_0^T \Big| D^\gamma \int_t^T R_t^sY(t;s)(x)\D s\Big|^2\D t\\
        \leq & C\Big( \sup_{t} \int_{t}^T \!\!\!\int_{\R} \left| D^\gamma G_{s,t}(x-y)\right||x-y|^{\beta}\D y\D s \Big)^2 \cdot\sup_{y}\E\int_0^T\sup_{t\leq s}\frac{|Y(t;s,y)-Y(t;s,x)|^2}{|x-y|^{2\beta}}\D s  \leq  C
        \Vert f\Vert_{\beta,\mathcal{L}^2}^2.
    \end{align*}
Thus, 
\begin{align}\label{eq 332}
    \Big\| \int_{\cdot}^T R_{\cdot}^s Y(\cdot;s)\D s\Big\|_{2,\mathcal{L}^2}\leq C\Vert f\Vert_{\beta,\mathcal{L}^2}.
\end{align}

Next, we derive the Hölder estimates of  
$$
\int_{\cdot}^T D^2 R_{\cdot}^s Y(\cdot;s)\D s,
$$
which needs subtle estimates. Denote $\eta=2|x-\bar{x}|$, then
    \begin{align*}
    D^2 \int_t^T\int_{\R}G_{s,t}(x-y)Y(t;s,y)\D y\D s - D^2 \int_t^T\int_{\R}G_{s,t}(\bar{x}-y)Y(t;s,y)\D y\D s= I_1-I_2+I_3-I_4
    , 
    \end{align*}
    where
    \begin{align*}
        I_1(t,x,\bar{x})&=\int_t^T \int_{|y-x|\leq \eta} D^2 G_{s,t}(x-y)\Big(Y(t;s,y)-Y(t;s,x)\Big)\D y\D s,\\
        I_2(t,x,\bar{x})&=\int_t^T \int_{|y-x|\leq \eta} D^2 G_{s,t}(\bar{x}-y)\Big(Y(t;s,y)-Y(t;s,\bar{x})\Big)\D y\D s,\\
        I_3(t,x,\bar{x})&=\int_t^T \int_{|y-x|> \eta} \Big( D^2 G_{s,t}(x-y)- D^2 G_{s,t}(\bar{x}-y)\Big)\Big(Y(t;s,y)-Y(t;s,x)\Big)\D y\D s,\\
        I_4(t,x,\bar{x})&=\int_t^T \int_{|y-x|> \eta} D^2 G_{s,t}(\bar{x}-y)\Big(Y(t;s,\bar{x})-Y(t;s,x)\Big)\D y\D s.
    \end{align*}
In view of Lemma \ref{all G estimate}~\eqref{sup G_2 estimate}, 
    \begin{align*}
        &\E\int_0^T|I_1(t,x,\bar{x})|^2\D t\\
        \leq & C\Big( \int_{|y-x|\leq \eta}\sup_{t\in[0,T]} \int_t^T\left|D^2G_{s,t}(x-y)\right||x-y|^{\beta}\D s\D y\Big)^2 \cdot \sup_{y}\E\int_0^T\sup_{t\leq s}\frac{|Y(t;s,y)-Y(t;s,x)|^2}{|y-x|^{2\beta}}\D s\\
        \leq & C|x-\bar{x}|^{2\alpha+2\beta-4}\cdot\Vert f \Vert_{\beta,\mathcal{L}^2}^2.
    \end{align*}
    Similarly, we have
     \begin{align*}
        &\E\int_0^T|I_2(t,x,\bar{x})|^2\D t\leq C|x-\bar{x}|^{2\alpha+2\beta-4}\cdot\Vert f \Vert_{\beta,\mathcal{L}^2}^2.
    \end{align*}
To prove the estimate of $I_3$, we notice that
   \begin{align*}
        \E\int_0^T|I_3(t,x,\bar{x})|^2\D t \leq & C\Big( \int_{|y-x|> \eta}\sup_{t\in[0,T]}\int_t^T \Big|D^2G_{s,t}(x-y)-D^2G_{s,t}(\bar{x}-y)\Big||x-y|^{\beta}\D s\D y\Big)^2\\
        &\times\sup_{y}\E\int_0^T\sup_{t\leq s}\frac{|Y(t;s,y)-Y(t;s,x)|^2}{|y-x|^{2\beta}}\D s.
    \end{align*}  
Furthermore, by Lemma \ref{all G estimate}~\eqref{sup G_2 estimate},
    \begin{align*}
         &\int_{|y-x|> \eta}\sup_{t\in[0,T]}\int_t^T \Big|D^2G_{s,t}(x-y)-D^2G_{s,t}(\bar{x}-y)\Big||x-y|^{\beta}\D s\D y\\
        \leq & |x-\bar{x}|\int_{|y-x|> \eta}\sup_{t\in[0,T]}\int_t^T\int_0^1\left|D^3G_{t,s}({x-y}+\rho(\bar{x} - x))\right||x-y|^{\beta}\D \rho \D s\D y \leq C|x-\bar{x}|^{\alpha+\beta-2}.
    \end{align*}
Next, we get
    \begin{align*}
        &\E\int_0^T|I_4(t,x,\bar{x})|^2\D t\\
        \leq & C|\bar{x} -x|^{2\beta} \Big( \int_{|y-x|> \eta}\sup_{t\in[0,T]} \int_t^T\left|D^2G_{s,t}(\bar{x}-y)\right| \D s\D y\Big)^2 \cdot \E\int_0^T\sup_{t\leq s}\frac{|Y(t;s,\bar{x})-Y(t;s,x)|^2}{|\bar{x} -x|^{2\beta}}\D s\\
        \leq & C|x-\bar{x}|^{2\alpha+2\beta-4}\cdot\Vert f \Vert_{\beta,\mathcal{L}^2}^2.
    \end{align*}
In summary, we have proved that that
\begin{align}\label{eq 365}
       \Big[ D^2 \int_{\cdot}^T R_{\cdot}^s Y(\cdot;s)\D s\Big]_{\alpha+\beta-2,\mathcal{L}^2}\leq C\Vert f\Vert_{\beta,\mathcal{L}^2}.
\end{align}
Then~\eqref{eq 308} is proved if we combine with~\eqref{eq 332} and~\eqref{eq 365}.

 Next, We prove \begin{align}\label{eq 516}
 \Vert R_{\cdot}^T p(\cdot) \Vert_{\alpha+\beta,\mathcal{L}^2}\leq C\Vert g\Vert_{\frac{\alpha}{2}+\beta,{L}^2}.
\end{align}  
With the similar calculation that derives~\eqref{eq 332}, we get 
\begin{align}\label{eq 527}
 \left\| R_{\cdot}^T p(\cdot) \right\|_{2,\mathcal{L}^2}\leq C\Vert g\Vert_{\frac{\alpha}{2}+\beta,{L}^2}.
\end{align}  
    Since $\alpha+\beta>2$, by Lemma \ref{pqYZ regular}, $p$ is continuous differentiable with $x$, therefore
    \begin{align}\nonumber
         {D^2 }R_t^T p (t)(x)=\int_{\R} {D}G_{T,t}(x-y){D p}(t,y)\D y.
    \end{align}
In the following, we consider
    \begin{align*}
   J_1+J_2+J_3-J_4= \int_{\R} \left(DG_{T,t}(x-y)-DG_{T,t}(\bar{x}-y)\right)D p(t,y)\D y,
\end{align*}
where 
\begin{align*}
& J_1(t,x,\bar{x})=\int_{|x-y|>\eta}\left(DG_{T,t}(x-y)-DG_{T,t}(\bar{x}-y)\right)\left(D p(t,y)-D p(t,x)\right)\D y,\\
 & J_2(t,x,\bar{x})=\int_{|x-y|>\eta} DG_{T,t}(\bar{x}-y) \left(D p(t,\bar{x})-D p(t,x)\right)\D y\\
 & J_3(t,x,\bar{x})=\int_{|x-y|\leq \eta}DG_{T,t}(x-y)\left(D p(t,y)-D p(t,x)\right)\D y,\\
 & J_4(t,x,\bar{x})=\int_{|x-y|\leq \eta}DG_{T,t}(\bar{x}-y)\left(D p(t,y)-D p(t,\bar{x})\right)\D y.
\end{align*}
Denote $a=\frac{\alpha}{2}+\beta-1$, and
 \begin{align*}
    A_1(t,x,\bar{x},y) &=\int_0^1 \left|D^2 G_{T,t}(\bar{x}-y+r(x-\bar{x}))\right|\D r|y-x|^{a}|x-\bar{x}|,\\
     A_2(t,x,y) &=|y-x|^{-a}\left|D p(t,y)-D p(t,x)\right|.
\end{align*}
Then by Cauchy's inequality, Lemma \ref{all G estimate}~\eqref{Lemma G 219} and Lemma \ref{pqYZ regular}, we have
\begin{align*}
     &\E\int_0^T |J_1(t,x,\bar{x})|^2\D t \leq \sup_{t,y}\E\left[A_2(t,x,y)^2\right]\int_0^T\Big(\int_{|x-y|>\eta}A_1(t,x,\bar{x},y)\D y\Big)^2\D t\\
     \leq& C |x-\bar{x}|^2\|p\|^2_{{\frac{\alpha}{2}+\beta},\,\mathcal{S}^2}\int_0^T\Big(\int_{|z|> \eta/2}\left|{D^2 G}_{T,t}(z)\right|(|z|^{a}+|\eta|^{a})\D y\Big)^2\D t\leq C |x-\bar{x}|^{2\alpha+2\beta-4}\|g\|^2_{\frac{\alpha}{2}+\beta,L^2}.
 \end{align*}
And it holds that by Lemma \ref{all G estimate}~\eqref{Lemma G 219}
\begin{align*}
     &\E\int_0^T |J_2(t,x,\bar{x})|^2\D t\leq \E\left[ \sup_{t\in[0,T]} \left| D p(t,\bar{x})-D p(t,x) \right|^2\right]\int_0^T\Big(\int_{|x-y|>\eta} | DG_{T,t}(\bar{x}-y) | \D y\Big)^2\D t\\
     \leq &C |x-\bar{x}|^{2a} \|p\|^2_{{\frac{\alpha}{2}+\beta},\,\mathcal{S}^2} \int_0^T\Big(\int_{|z|>\eta/2} | DG_{T,t}(z) | \D z\Big)^2\D t\leq C |x-\bar{x}|^{2\alpha+2\beta-4}\|g\|^2_{\frac{\alpha}{2}+\beta,L^2}.
 \end{align*}
Similarly, we have
\begin{align*}
    \E\int_0^T |J_3(t,x,\bar{x})|^2\D t\leq &C\sup_{t,y}\E\left[A_2(t,x,y)^2\right]\int_0^T\Big(\int_{|x-y|\leq \eta}D G_{T,t}(x-y)|x-y|^{a}\D y\Big)^2\D t \\
    \leq& C |x-\bar{x}|^{2\alpha+2\beta-4}\|g\|^2_{\frac{\alpha}{2}+\beta,L^2}\,,
\end{align*}
and
\begin{align*}
    &\E\int_0^T |J_4(t,x,\bar{x})|^2\D t\leq C \sup_{t,y} \E\left[\frac{\left|D p(t,y)-D p(t,\bar{x})\right|}{|y-\bar{x}|^{2a}}\right]\int_0^T\Big(\int_{|x-y|\leq \eta}D G_{T,t}(\bar{x}-y) |\bar{x}-y|^{a}\D y\Big)^2\D t \\
    \leq &C \|p\|^2_{{\frac{\alpha}{2}+\beta},\,\mathcal{S}^2} \int_0^T\Big(\int_{|z|\leq 3\eta/2}D G_{T,t}(z) |z|^{a}\D y\Big)^2\D t
    \leq C |x-\bar{x}|^{2\alpha+2\beta-4}\|g\|_{\frac{\alpha}{2}+\beta,L^2}\,.
\end{align*}
Thus, 
\begin{align}\label{eq 585}
     \left[ D^2 R_{\cdot}^T p(\cdot)\right]_{\alpha+\beta-2,\mathcal{L}^2}\leq C\Vert g\Vert_{\frac{\alpha}{2}+\beta,{L}^2}.
\end{align}
Combing~\eqref{eq 527} and~\eqref{eq 585}, we prove~\eqref{eq 516}. Therefore we finally complete our proof.
\end{proof}

We have the following regularity result for the classical solution. 
\begin{lemma}\label{uv Hölder}
Let Assumptions~\ref{Baisc assumption} and~\ref{Simple assumption} hold. For the solution $(u,v)$ to BSPDE~\eqref{linear-case}, we have
    \begin{align*}
         \Vert u\Vert_{\alpha+\beta,\mathcal{L}^2}+\Vert u\Vert_{\beta,\mathcal{S}^2}+\Vert v\Vert_{\beta,\mathcal{L}^2}\leq C\big(\Vert g\Vert_{\frac{\alpha}{2}+\beta,{L}^2}+\Vert f\Vert_{\beta,\mathcal{L}^2}\big) ,
    \end{align*}
     where $C=C(\alpha,\beta,K,T,m)$.
\end{lemma}
\begin{proof}
    By Ito's formula, it holds that
    \begin{align}\label{eq 395}
    \D u^2(t,x)=&\left[2 a(t) \left(u (-\Delta )^{\frac{\alpha}{2}}u\right)-2uf-2\sigma(t)uv+v^2\right](t,x)\D t+2(uv)(t,x)\D W_t.
    \end{align}
    Hence we get
    \begin{align*}
        \E\int_0^T |v(t,x)|^2\D t&\leq C\E\left[|g(x)|^2\right]+C\E\int_0^T \left[\left|f(t,x)\right|^2+\left|u(t,x)\right|^2+\left|(-\Delta )^{\frac{\alpha}{2}}u(t,x)\right|^2\right]\D t.
    \end{align*}
    Therefore 
    \begin{align*}
        \|v\|_{0,\mathcal{L}^2}\leq \|g\|_{0,L^2}+ \|u\|_{0,\mathcal{L}^2}+\|(-\Delta )^{\frac{\alpha}{2}}u\|_{0,\mathcal{L}^2}+\|f\|_{0,\mathcal{L}^2}.
    \end{align*}
    According to Lemma \ref{delta Hölder},
    \begin{align*}
        \|(-\Delta )^{\frac{\alpha}{2}}u\|_{0,\mathcal{L}^2}\leq  \|(-\Delta )^{\frac{\alpha}{2}}u\|_{\frac{\beta}{2},\mathcal{L}^2}\leq C\|u\|_{\alpha+\frac{\beta}{2},\mathcal{L}^2}.
    \end{align*}
   It follows that
   \begin{align}
   \nonumber      \|v\|_{0,\mathcal{L}^2}&\leq C\left(\|g\|_{0,L^2}+ \|u\|_{0,\mathcal{L}^2}+\|u\|_{\alpha+\frac{\beta}{2},\mathcal{L}^2}+\|f\|_{0,\mathcal{L}^2}\right)\\
     \label{eq 498}    &\leq C\left( \|g\|_{\frac{\alpha}{2}+\beta,L^2}+ \|u\|_{\alpha+\beta,\mathcal{L}^2}+\|f\|_{\beta,\mathcal{L}^2}\right).
   \end{align}
Similarly, we have
\begin{align*}
          \E\int_0^T |v(t,x)-v(t,y)|^2\D t\leq& C\E\left[|g(x)-g(y)|^2\right]+C\E\int_0^T \left[\left|f(t,x)-f(t,y)\right|^2\right]\D t  \\
           +&C\E\int_0^T \left[\left|u(t,x)-u(t,y)\right|^2+\left|(-\Delta )^{\frac{\alpha}{2}}u(t,x)-(-\Delta )^{\frac{\alpha}{2}}u(t,y)\right|^2\right]\D t. 
\end{align*}
       It yields that
   \begin{align*}
         \|v\|_{\beta,\mathcal{L}^2}&\leq C \left(\|g\|_{\beta,L^2}+ \|u\|_{\beta,\mathcal{L}^2}+\|(-\Delta )^{\frac{\alpha}{2}}u\|_{\beta,\mathcal{L}^2}+\|f\|_{\beta,\mathcal{L}^2}\right)\\
         &\leq C \left(\|g\|_{\frac{\alpha}{2}+\beta,L^2}+ \|u\|_{\alpha+\beta,\mathcal{L}^2}+\|f\|_{\beta,\mathcal{L}^2}\right) \leq C \left(\|g\|_{\frac{\alpha}{2}+\beta,L^2} + \|f\|_{\beta,\mathcal{L}^2}\right).
   \end{align*}

By BDG inequality and Gronwall inequality, we have
\begin{align*}
    \E\left[\max_{t\in[0,T]}|u(t,x)|^2\right]\leq C\E\left[|g(x)|^2\right]+C\E\int_0^T \left[|v(t,x)|^2+|(-\Delta)^{\frac{\alpha}{2}}u(t,x)|^2+|f(t,x)|\right]\D t.
\end{align*}
and 
\begin{align*}
    \E\left[\max_{t\in[0,T]}|u(t,x)-u(t,y)|^2\right]\leq &C\E\left[|g(x)-g(y)|^2\right]+C\E\int_0^T \left[|v(t,x)-v(t,y)|^2\right]\D t\\
    +&C\E\int_0^T\left[|(-\Delta)^{\frac{\alpha}{2}}u(t,x)-(-\Delta)^{\frac{\alpha}{2}}u(t,y)|^2+|f(t,x)-f(t,y)|^2\right] \D t.
\end{align*}
Therefore
\begin{align*}
    \|u\|_{\beta,\mathcal{L}^2}\leq &C\left(\|g\|_{\beta,{L}^2}+\|v\|_{\beta,\mathcal{L}^2}+\|(\Delta)^{\frac{\alpha}{2}} u\|_{\beta,\mathcal{L}^2}+\|f\|_{\beta,\mathcal{L}^2}\right)\leq C(\|g\|_{\frac{\alpha}{2}+\beta,{L}^2}+\|f\|_{\beta,{L}^2}).
\end{align*}
\end{proof}

\subsection{Fractional BSPDE with coefficients $a(t,x)$ and $\sigma(t,x)$}
Next we consider the case where the coefficients $a(\cdot)$ and $\sigma(\cdot)$ depend on space-time variable $(t,x)$, and $(b,c)\neq 0$. The main difficulty is that there is no explicit formulation for the solution in this case. But with the help of techniques of the freezing coefficients method and the continuation method, we can get the wellposedness and regularity for the general fractional BSPDE~\eqref{linear-case 2}, which is Theorem \ref{Theorem genral}.

\subsubsection{Hölder estimate: freezing coefficients method}
In order to apply the continuation method for studying the wellposedness and regularity, we first need to derive the priori Hölder estimation for the solution $(u, v)$. Now we use the freezing coefficients method give the Hölder estimate of the solution to BSPDE~\eqref{linear-case 2}. 

Choose $\phi\in C_c^{\infty}(\R)$ such that
$$
    0\leq \phi\leq 1\quad\text{and} \quad \phi(x)=\begin{cases}
        1,\quad |x|\leq 1;\\
        0,\quad |x|>2.
    \end{cases}
$$
For $z\in{\R}$ and $\theta>0$, we define
\begin{align*}
    \phi_{\theta}^z(x):=\phi\left(\frac{x-z}{\theta}\right).
\end{align*}
For each $\beta \in[0,1)$ and $k\in\mathbb{N}$, it can be verified that
\begin{align}\label{eq 535}
  \left[D^k \phi_{\theta}^z\right]_{\beta}\leq \theta^{-k-\beta} \left[D^k \phi \right]_{\beta},\quad\quad \left[(-\Delta)^{\frac{\alpha}{2}}\phi_{\theta}^z\right]_{\beta}\leq \theta^{-\alpha-\beta} \left[(-\Delta)^{\frac{\alpha}{2}}\phi \right]_{\beta}.
\end{align}

The following two lemmas give the localization properties and interpolation inequalities of Hölder norms, which can be found from Lemma 4.1 and Lemma 2.1 in~\cite{tang2016cauchy}. 
\begin{lemma}\label{tuncture bound}
    Let $h\in C^{m+\gamma}(\R,\mathcal{L}^2_{\mathbb{F}}(0,T))$ with $m=0,1,2$ and $\gamma\in(0,1)$. Then there is a positive constant $C(\theta,\gamma)$ such that
    \begin{align*}
        \|h\|_{m+\gamma,\mathcal{L}^2}\leq 2\sup_{z\in{\R}} \| h \phi_{\theta}^z \|_{m+\gamma,\mathcal{L}^2}+C(\theta,\gamma)\|h\|_{0,\mathcal{L}^2}.
    \end{align*}
\end{lemma}

\begin{lemma}\label{Hölder bound}
   For each $\varepsilon>0$ and $\gamma\in(0,1)$, there is $C=C(\varepsilon,\gamma)>0$ such that for all $\Phi \in C^{2+\gamma}(\R,\mathcal{L}_{\mathbb{F}}^2(0,T))$ 
   \begin{align*}
      \left[\Phi\right]_{1+\gamma,\mathcal{L}^2}+ \|\Phi\|_{2,\mathcal{L}^2}\leq \varepsilon \left[\Phi\right]_{2+\gamma,\mathcal{L}^2}+C\|\Phi\|_{0,\mathcal{L}^2}.
   \end{align*}
\end{lemma}

Now, we apply the freezing coefficients method to derive the Hölder estimates for the solution to BSPDE (\ref{linear-case 2}).
\begin{theorem}\label{prior Hölder}
   Let Assumptions \ref{Baisc assumption} and \ref{general assumption} hold. If $(u,v)$ solves the fractional BSPDE~\eqref{linear-case 2}, then $u$ is in $C_{\mathcal{S}^2}^{\beta}\cap C_{\mathcal{L}^2}^{\alpha+\beta}$ and $v$ is in $C_{\mathcal{L}^2}^{\beta}$ such that
\begin{align*}
        \Vert u\Vert_{\alpha+\beta,\mathcal{L}^2}+ \Vert u\Vert_{\beta,\mathcal{S}^2}+\Vert v\Vert_{\beta,\mathcal{L}^2}\leq C\big(\Vert g\Vert_{\frac{\alpha}{2}+\beta,{L}^2}+\Vert f\Vert_{\beta,\mathcal{L}^2}\big),
\end{align*}
where $C=C(\alpha,\beta,\Lambda,T,m,K)$.
\end{theorem}
   
\begin{proof}
For $t\in[0,T]$ and $x\in\R$, we define 
\begin{align*}
    u_\theta^z(t,x)=u(t,x)\phi_\theta^z(x),\quad v_\theta^z(t,x)=v(t,x)\phi_\theta^z(x)\quad\text{and}\quad  g_\theta^z(x)=g(x)\phi_\theta^z(x).
\end{align*}
Note that
$$
u_\theta^z(T,x)=g_\theta^z(x).
$$
And by Ito's formula, we have
\begin{align}\label{eq 640}
    \D u_\theta^z(t,x)=\left[a(t,z)(-\Delta)^{\frac{\alpha}{2}}u_\theta^z(t,x)-f_{\theta}^z(t,x)-\sigma(t,z)v_\theta^z(t,x)\right]\D t+v_{\theta}^z(t,x)\D W_t,
\end{align}
where $f_{\theta}^z(t,x) :=\sum_{i=1}^6 K_i(t,x,z,\theta)$ with 
\begin{align*}
    &K_1(t,x,z,\theta):=a(t,z)\left((-\Delta)^{\frac{\alpha}{2}}u_\theta^z(x)-(-\Delta)^{\frac{\alpha}{2}}u(t,x)\cdot \phi_\theta^z(x)\right),\\
    &K_2(t,x,z,\theta):= (a(t,z)-a(t,x))(-\Delta)^{\frac{\alpha}{2}} u(t,x)\cdot\phi_\theta^z(t,x),\\
    &K_3(t,x,z,\theta):=(\sigma(t,x)-\sigma(t,z))v_\theta^z(t,x),\quad\quad 
    K_4(t,x,z,\theta):= b(t,x){D u}{ }(t,x)\cdot\phi_{\theta}^z(x),\\ &K_5(t,x,z,\theta):=c(t,x)u_{\theta}^z(t,x),\quad\quad  K_6(t,x,z,\theta):=f(t,x)\phi_{\theta}^z(x).
\end{align*}
By Lemma \ref{uv Hölder} and Lemma \ref{tuncture bound}, we have
\begin{align}\label{eq 590}
    \|u\|_{\alpha+\beta,\mathcal{L}^2}+ \|v\|_{\beta,\mathcal{L}^2}\leq C\left(\sup_z\|g^z_{\theta}\|_{\frac{\alpha}{2}+\beta,{L}^2}+\sup_z\|f_{\theta}^z\|_{\beta,\mathcal{L}^2}\right)+C(\theta)\left(\|u\|_{0,\mathcal{L}^2}+\|v\|_{0,\mathcal{L}^2}\right).
\end{align}
Next, we estimate $\|g^z_{\theta}\|_{\frac{\alpha}{2}+\beta,{L}^2}$ and $\|K_i(\cdot,\cdot,z,\theta)\|_{\beta,\mathcal{L}^2}\, ( 1\leq i\leq 6)$, one by one.

For $\|g^z_{\theta}\|_{\frac{\alpha}{2}+\beta,{L}^2}$, it follows from~\eqref{eq 535} that
\begin{align*}
    &\|g_{\theta}^z\|_{\frac{\alpha}{2}+\beta,L^2}=\|\phi_\theta^zg\|_{0,L^2}+\|D(\phi_\theta^zg)\|_{0,L^2}+[\phi_\theta^zg]_{\frac{\alpha}{2}+\beta,L^2}\\[0.2cm]
    \leq &C\left(\theta^{-1}+\theta^{-\frac{\alpha}{2}-\beta}\right)\left\|g\right\|_{0,L^2}+C\left(1+\theta^{-\frac{\alpha}{2}-\beta+1}\right)\left\|g\right\|_{1,L^2}+C\theta^{-1}\left[g\right]_{\frac{\alpha}{2}+\beta-1,L^2}+C\left[g\right]_{\frac{\alpha}{2}+\beta,L^2}\\
    \leq &C(\theta)\|g\|_{\frac{\alpha}{2}+\beta,L^2}.
\end{align*}

For $\|K_1(\cdot,\cdot,z,\theta)\|_{\beta,\mathcal{L}^2}$, we first consider the case $\alpha=2$, in which it holds that
$$
K_1(t,x,z,\theta):=a(t,z)\left(2 D u(t,x)\cdot D\phi_\theta^z(x) + u(t,x)\cdot D^2\phi_\theta^z(x)\right),
$$
which gives that $\|K_1(\cdot,\cdot,z,\theta)\|_{\beta,\mL^2} \leq C(\theta) \|u\|_{1+\beta,\mL^2}$ by~\eqref{eq 535}. For the case $1<\alpha<2$, we use the pointwise formulation~\eqref{pointwise formula} to get
\begin{align*}
    &-(-\Delta)^{\frac{\alpha}{2}}u_\theta^z(t,x)+(-\Delta)^{\frac{\alpha}{2}}u(t,x)\phi_\theta^z(x)\\
    =&-\int_{|y|>1}\frac{1}{|y|^{1+\alpha}}\left[u_\theta^z(t,x)-u_\theta^z(t,x+y)-\phi_{\theta}^z(x)(u(t,x)-u(t,x+y))\right]\D y\\
    &-\int_{|y|\leq 1}\frac{1}{|y|^{1+\alpha}}\left[u_\theta^z(t,x)-u_\theta^z(t,x+y)+Du_\theta^z(t,x)y\right]\D y\\
      &+\int_{|y|\leq 1}\frac{1}{|y|^{1+\alpha}}\left[\phi_{\theta}^z(x)\left(u(t,x)-u(t,x+y)+Du(t,x)y\right)\right]\D y\\
        =&\int_{|y|>1}\frac{u(t,x+y)}{|y|^{1+\alpha}}\left[\phi_\theta^z(x+y)-\phi_\theta^z(x)\right]\D y\\
    &+\int_{|y|\leq 1}\frac{1}{|y|^{1+\alpha}}\left[ u(t,x+y)-u(t,x)\right]\left[ \phi_\theta^z(x+y)-\phi_\theta^z(x)\right]\D y\\
      &+\int_{|y|\leq 1}\frac{1}{|y|^{1+\alpha}}u(t,x)\left[\phi_{\theta}^z(x+y)-\phi_{\theta}^z(x)-D \phi_\theta^z(x)y\right]\D y\\
    =&\int_0^1 \int_{|y|>1}\frac{y}{|y|^{1+\alpha}}u(t,x+y)D\phi_{\theta}^z(x+\lambda y)\D y\D \lambda\\
    &+\int_0^1 \int_0^1 \int_{|y|\leq 1}\frac{1}{|y|^{\alpha-1}}Du(t,x+\lambda y)D\phi_{\theta}^z(x+\mu y)\D y\D \lambda\D\mu\\
    &+\int_0^1\int_{|y|\leq 1}\frac{1-\lambda}{|y|^{\alpha-1}}u(t,x)D^2\phi_\theta^z(x+\lambda y)\D y\D \lambda:=\sum_{i=1}^3 M_i(t,x),
\end{align*}
where $M_i(t,x),\,i=1,2,3$ denote the $i$-th term in the above summation. By~\eqref{eq 535}, we get
$$
\|M_1 \|_{0,\mathcal{L}^2}^2 \leq  \frac{C}{\theta^2} \|u \|_{0,\mathcal{L}^2}^2,\quad\quad \|M_2 \|_{0,\mathcal{L}^2}^2 \leq \frac{C}{\theta^2} \| Du \|_{0,\mathcal{L}^2}^2, \quad\quad \|M_3 \|_{0,\mathcal{L}^2}^2 \leq \frac{C}{\theta^4} \| u \|_{0,\mathcal{L}^2}^2,
$$
and
\begin{align*}
      \sum_{i=1}^3 [M_i]_{\beta,\mathcal{L}^2}^2 \leq \left( \frac{C}{\theta^2} + \frac{C}{\theta^4} \right)[u ]_{\beta,\mathcal{L}^2}^2 + \frac{C}{\theta^2}[Du ]_{\beta,\mathcal{L}^2}^2 + \left( \frac{C}{\theta^{2+2\beta}} + \frac{C}{\theta^{4+2\beta}} \right)\|u\|_{0,\mathcal{L}^2}^2 +\frac{C}{\theta^{2+2\beta}}\|D u\|_{0,\mathcal{L}^2}^2.
\end{align*}
Thus for all $\alpha \in(1,2]$, it holds that
$$
\|K_1(\cdot,\cdot,z,\theta)\|_{\beta,\mL^2} \leq C(\theta) \|u\|_{1+\beta,\mL^2}.
$$
Taking $\gamma=\alpha+\beta-2$ in Lemma \ref{Hölder bound}, we know that
$$
\|K_1(\cdot,\cdot,z,\theta)\|_{\beta,\mL^2} \leq C(\theta) \|u\|_{2,\mL^2} \leq \varepsilon C(\theta)\|u\|_{\alpha+\beta,\mL^2}+C(\theta)\|u\|_{0,\mL^2}.
$$

To estimate $\|K_2(\cdot,\cdot,z,\theta)\|_{\beta,\mathcal{L}^2}$, we denote $\left[\cdot\right]_{m+\beta,\mathcal{L}^2, A}$ and $\left\|\cdot\right\|_{m+\beta,\mathcal{L}^2,A}$ to be the semi-norm and norm of functions on subset $A\subset \R$ instead of on the whole space $\R$. By Lemma \ref{Hölder bound} and Lemma \ref{delta Hölder}, we have
\begin{align*}
    & \|K_2(\cdot,\cdot,z,\theta)\|_{\beta,\mL^2}= \|K_2(\cdot,\cdot,z,\theta)\|_{\beta,\mL^2,B_{2\theta}(z)}\\
\leq& [a(\cdot,\cdot)-a(\cdot,z)]_{0,\mL^{\infty},B_{2\theta}(z)}\left[(-\Delta)^{\frac{\alpha}{2}}u\right]_{\beta,\mL^2}\\
    &+\left([a(\cdot,\cdot)-a(\cdot,z)]_{\beta,\mL^{\infty}} + [a(\cdot,\cdot)-a(\cdot,z)]_{0,\mL^{\infty},B_{2\theta}(z)}[\phi_\theta^z]_{\beta} \right)\|(-\Delta)^{\frac{\alpha}{2}}u\|_{0,\mL^2}\\
    \leq & \theta^{\beta} \|u\|_{\alpha+\beta,\mL^2} +C\|u\|_{2,\mL^2}
     \leq C \|u\|_{\alpha+\beta,\mL^2}\left(\theta^\beta +\varepsilon \right) +C(\varepsilon)\|u\|_{0,\mL^2}.
\end{align*}
Similarly, we can derive that
\begin{align*}
  & \|K_3(\cdot,\cdot,z,\theta)\|_{\beta,\mL^2}\leq C\theta^\beta\|v\|_{\beta,\mL^2}+C(\theta)\|v\|_{0,\mL^2},\\
 &  \|K_4(\cdot,\cdot,z,\theta)\|_{\beta,\mL^2}+ \|K_5(\cdot,\cdot,z,\theta)\|_{\beta,\mL^2}\leq \varepsilon C(\theta)[u]_{\alpha+\beta,\mL^2}+C(\varepsilon,\theta)\|u\|_{0,\mL^2},\\
 &  \|K_6(\cdot,\cdot,z,\theta)\|_{\beta,\mL^2}\leq C(\theta)\|f\|_{\beta,\mL^2}.
\end{align*}
In view of~\eqref{eq 590}, we take $\theta,\varepsilon$ small enough such that 
\begin{align} \label{eq 642}  \|u\|_{\alpha+\beta,\mL^2}+\|v\|_{\beta,\mL^2}\leq C\left(\|g\|_{\frac{\alpha}{2}+\beta,\mL^2}+\|f\|_{\beta,\mL^2}+\|u\|_{0,\mL^2}+\|v\|_{0,\mL^2}\right).
\end{align}

With the Ito's formula and the similar calculation that leads to~\eqref{eq 498}, we have
\begin{align*}
    \|v\|_{0,\mL^2}\leq C\left(\|g\|_{0,L^2}+\|u\|_{2,\mL^2}+\|f\|_{0,\mL^2}\right)\leq \varepsilon [u]_{\alpha+\beta,\mathcal{L}^2}+C(\varepsilon)\left(\|g\|_{0,L^2}+\|u\|_{0,\mL^2}+\|f\|_{0,\mL^2}\right).
\end{align*}
Taking $\varepsilon$ small enough, by~\eqref{eq 642} we have
\begin{align}\label{eq 650}
    \|u\|_{\alpha+\beta,\mL^2}+\|v\|_{\beta,\mL^2}\leq C\left(\|g\|_{\frac{\alpha}{2}+\beta,\mL^2}+\|f\|_{\beta,\mL^2}+\|u\|_{0,\mL^2}\right).
\end{align}
Define $\mathcal{L}_{\mathbb{F}}^2(t,T)$ to be the Banach space of real-valued $\mathbb{F}$-adapted process $f$ with the finite norm
\begin{align*}
    \|f\|_{\mathcal{L}_{\mathbb{F}}^2(t,T)}:=
        &\E\Big[\int_t^T |f(s)|^2\D s\Big]^{\frac{1}{2}}<\infty,
        \end{align*}
and simply denote $\|u\|_{\gamma,\mL^2,t}:=\|u\|_{\gamma,\mathcal{L}_{\mathbb{F}}^2(t,T)}$. Using Ito's formula, we get
\begin{align*}
   \sup_{x} \E\left[|u(t,x)|^2\right]\leq& C\left(\|u\|^2_{\alpha+\beta,\mL^2,t}+\|v\|^2_{0,\mL^2,t}+\|g\|_{0,L^2}^2+\|f\|^2_{0,\mL^2}\right)\\
   \leq &C\left(\|u\|^2_{0,\mL^2,t}+\|g\|^2_{\frac{\alpha}{2}+\beta,L^2}+\|f\|^2_{\beta,\mL^2}\right).
\end{align*}
With the help of backward Gronwall's inequality, we have
\begin{align*}
    \|u\|_{0,\mL^2}\leq C\left(\|g\|_{\frac{\alpha}{2}+\beta,L^2}+\|f\|_{\beta,\mL^2}\right).
\end{align*}
Along with~\eqref{eq 650}, it holds that
\begin{align*}
 \|u\|_{\alpha+\beta,\mL^2}+\|v\|_{\beta,\mL^2}\leq C\left(\|g\|_{\frac{\alpha}{2}+\beta,\mL^2}+\|f\|_{\beta,\mL^2}\right).   
\end{align*}

Using Ito's formula again, we have
\begin{align*}
    \|u\|_{\beta,\mathcal{S}^2}\leq C\left(\|g\|_{\beta,L^2}+\|u\|_{\alpha+\beta,\mL^2}+\|v\|_{\beta,\mL^2}+\|f\|_{\beta,\mL^2}\right)\leq C\left(\|g\|_{\frac{\alpha}{2}+\beta,L^2}+\|f\|_{\beta,\mL^2}\right),
\end{align*}
which completes the proof.
\end{proof}
\begin{remark}
    It should be pointed out that, since $(-\Delta)^{\frac{\alpha}{2}}$ is not a local operator, most classical localization methods can not work out here. However, in this paper we can still get the Hölder estimates by successfully applying the freezing coefficients method, which can help us to use local properties of the solution to get the global estimates.
\end{remark}

\subsubsection{Wellposedness: continuation method}
To use the continuation method, we first define the operators $L$ and $M$ as
    $$
    Lu:=-a(-\Delta)^{\frac{\alpha}{2}}u+b{D u}+cu, \quad \quad Mv:=\sigma v,
    $$
    and for $\tau\in[0,1]$, we define the mixed operators
    $$
    L_{\tau}u:=\tau Lu-(1-\tau) (-\Delta)^{\frac{\alpha}{2}} u,\quad\quad M_{\tau}v:=\tau Mv.  
    $$
    Next, define the Banach space
    $$
    \mathcal{S}^{\alpha,\beta}:= \left[ C^{\alpha+\beta}(\R,\mathcal{L}^2_{\mathbb{F}}(0,T))\cap C^{\beta}(\R,\mathcal{S}^2_{\mathbb{F}}(0,T)) \right] \times C^\beta(\R,\mathcal{L}^2_{\mathbb{F}}(0,T))
    $$
    with the norm $\|(u,v)\|_{ \mathcal{S}^{\alpha,\beta}}:=\|u\|_{\alpha+\beta,\mathcal{L}^2}+\|u\|_{\beta,\mathcal{S}^2}+\|v\|_{\beta,\mathcal{L}^2}$, and define the Banach space
    $$
    \quad  \mathcal{B}^{\alpha,\beta}:=C^\beta(\R,\mathcal{L}^2_{\mathbb{F}}(0,T))\times C^{\frac{\alpha}{2}+\beta}(\R,L^2(\Omega))
    $$
     with the norm $\|(f,g)\|_{ \mathcal{B}^{\alpha,\beta}}:=\|f\|_{\beta,\mathcal{L}^2}+\|g\|_{\frac{\alpha}{2}+\beta,L^2}$. 
We are now ready to use the continuation method to prove Theorem \ref{Theorem genral}.
\begin{proof}[Proof of Theorem \ref{Theorem genral}]
    
For $\tau_0>0$ and $(f,g)\in  \mathcal{B}^{\alpha,\beta}$, we consider the following BSPDE
\begin{align}\label{fix-point}
\left\{
\begin{aligned}
  -\D U(t,x)=&\Big[L_\tau U +M_\tau V+f_{\tau_0}(u,v)\Big](t,x) \D t-V(t,x)\D W_t,\quad  (t,x)\in [0,T)\times \mathbb{R},\\
    U(T,x)=&\,\,g(x),\qquad\qquad\qquad\qquad\qquad\qquad\qquad\quad\qquad\qquad\qquad x\in\R,   
\end{aligned}
\right.
\end{align}
   where $f_{\tau_0}(u,v):=f+\tau_0(L_1-L_0)u+\tau_0 M_1v $.
    
    For all $(u,v)\in  \mathcal{S}^{\alpha,\beta}$, it follows that $f_{\tau_0}(u,v)\in C^\beta(\R,\mathcal{L}^2_{\mathbb{F}}(0,T))$. Notice when $\tau=0$, Theorem \ref{Thm1} tells us that BSPDE~\eqref{fix-point} has unique solution $(U_0,V_0)\in  \mathcal{S}^{\alpha,\beta}$. Consider the map $\Pi_0:  \mathcal{S}^{\alpha,\beta}\longrightarrow  \mathcal{S}^{\alpha,\beta}$, $\Pi_0(u,v):=(U_0,V_0)$. 
    
    Now we show that for small enough  $\tau_0$, the map $\Pi_0$ has a fixed point. To see this, we consider $(u^i,v^i)\in \mathcal{S}^{\alpha,\beta}$ and $\Pi_0(u^i,v^i):=(U^i_0,V^i_0)$, $i=1,2$. By Theorem \ref{prior Hölder}, we have
    \begin{align*}
        &\|(U^2_0-U^1_0,V^2_0-V^1_0)\|_{ \mathcal{S}^{\alpha,\beta}}\leq C\|f_{\tau_0}(u^2,v^2)-f_{\tau_0}(u^1,v^1)\|_{\beta,\mathcal{L}^2}\\
        \leq & C \tau_0 \left(\|(L_1-L_0)(u^2-u^1)\|_{\beta,\mathcal{L}^2}+\|M_1(v^2-v^1)\|_{\beta,\mathcal{L}^2}\right) \leq  C\tau_0 \|(u^2-u^1,v^2-v^1)\|_{ \mathcal{S}^{\alpha,\beta}}.
    \end{align*}
    Thus, there exists a constant $\bar{\tau}=\bar{\tau} (\alpha,\beta,\Lambda,T,m,K)$ such that when $\tau_0\in[0, \bar{\tau}]$, the map $\Pi_0$ has a fixed point $(U_0,V_0)$. 
    
    In other words, for each $\tau\in[0,1\wedge \bar{\tau}]$, the following BSPDE
    \begin{equation}\label{eq2407191}
    -\D U(t,x)=(L_\tau U +M_\tau V+f)(t,x) \D t-V(t,x)\D W_t,\quad\quad U(T,x)=g(x),
    \end{equation}
    has a solution $(U_{\tau},V_{\tau})\in \mathcal{S}^{\alpha,\beta}$, by noting that
    $$
    L_{\tau_1} U +M_{\tau_1} V+f_{\tau_2}(U,V) = L_{\tau_1+ \tau_2} U +M_{\tau_1 + \tau_2} V +f,
    $$
    for all $\tau_1, \tau_2\in[0,1]$ with $0\leq \tau_1+ \tau_2\leq 1$.
    
    If $\bar{\tau} \geq 1$, the existence of solution can be immediately derived by taking $\tau=1$. If $ \bar{\tau} <1$, we can continue the above procedure for the new map $\Pi_{\bar{\tau}}:\mathcal{S}^{\alpha,\beta}\longrightarrow  \mathcal{S}^{\alpha,\beta}$ with $\Pi_{\bar{\tau}}(u,v):=\left(U_{\bar{\tau}},V_{\bar{\tau}}\right)$, where $\left(U_{\bar{\tau}},V_{\bar{\tau}}\right)$ is the solution to BSPDE~\eqref{fix-point} when $\tau = \bar{\tau}$. Then with the analogous argument using Theorem \ref{prior Hölder}, we can show that for each $\tau\in[0,1\wedge 2\bar{\tau}]$, the BSPDE~\eqref{eq2407191} has a solution $(U_{\tau},V_{\tau})$. Repeating the procedure for finite steps, we finally derive the existence of the solution to BSPDE~\eqref{eq2407191} when $\tau=1$, which is BSPDE~\eqref{linear-case}. The uniqueness of solution can be directly derived by Theorem~\ref{prior Hölder}.
\end{proof}

\appendix
\section*{Appendix}

\renewcommand{\thesubsection}{\Alph{subsection}}
\setcounter{equation}{0}
\setcounter{lemma}{0}
\setcounter{remark}{0}
\setcounter{theorem}{0}
\setcounter{corollary}{0}
\renewcommand{\thetheorem}{A.\arabic{theorem}}
\renewcommand{\theequation}{A.\arabic{equation}}
\renewcommand{\thelemma}{A.\arabic{lemma}}
\renewcommand{\theremark}{A.\arabic{remark}}
\renewcommand{\thecorollary}{A.\arabic{corollary}}

\subsection{Well-posedness of solution in Soblev spaces}\label{Appendix A}
This subsection is devoted to show the well-posedness of solution to BSPDE~\eqref{linear-case} in Soblev space. According to Remark~\ref{remark664}, we know the Fourier transformation of the solution can be written as
\begin{align}\label{eq2494}
    \hat{u}(t,\xi)=\E\left[\Gamma_{T,t}(\xi)\hat{g}(\xi)+\int_t^T\Gamma_{s,t}(\xi)\hat{f}(s,\xi)\D s \bigg| \mathcal{F}_t\right],
\end{align}
with \begin{align*}
    \Gamma_{s,t}(\xi)=\exp\left(-\int_t^s a(r)|\xi|^{\alpha}\D r\right)\exp\left(\int_t^s \sigma(r)\D W_r-\int_t^s\frac{1}{2}|\sigma(r)|^2 \D r\right):=\hat{\Gamma}_{s,t}(\xi)\,\bar{\Gamma}_{s,t}.
\end{align*}
\begin{lemma}\label{lemma1385}
 Let Assumptions~\ref{Baisc assumption} and~\ref{Simple assumption} hold. If $(f,g)\in \mathbb{H}^\gamma \times L^2(\Omega,H^\gamma)$, then the inverse Fourier transformation of $\hat{u}$ defined in~\eqref{eq2494} satisfies that
    \begin{align*}
    \E\left[\sup_{t\in[0,T]}\|u(t,\cdot)\|^2_{H^{\frac{\alpha}{2}+\gamma}}\right] + \E\left[\int_0^T\|u(t,\cdot)\|^2_{H^{\alpha+\gamma}}\D t \right] \leq C\,\E\left[\|g\|_{H^{\frac{\alpha}{2}+\gamma}}^2\right] +C\, \E\left[ \int_0^T \|f(t,\cdot)\|^2_{H^{\gamma}}\D t \right],
        \end{align*}
     where the constants $\gamma>0$ and $C=C(\alpha,\beta,\gamma,K,m,T)$.  
\end{lemma}

\begin{proof}
Since $\sigma$ is bounded, it holds that for any $p>1$
\begin{align}\label{eq1384}
    \sup_{0\leq t\leq T}\E\left[\sup_{ t<s\leq T}|\bar{\Gamma}_{s,t}|^p\Big|\mathcal{F}_t\right]<\infty.
\end{align}

It follows from the explicit structure~\eqref{eq2494} that
\begin{align*}
    \E\left[\sup\limits_{t\in[0,T]}\|u(t,\cdot)\|^2_{H^{\frac{\alpha}{2}+\gamma}}\right]\leq &C\int_\R\left(1+|\xi|^{\alpha+2\gamma}\right) \E \left[ \sup_{0\leq t\leq T}\left|\E\left[\Gamma_{T,t}(\xi)\hat{g}(\xi) \big| \mathcal{F}_t\right]\right|^2 \right] \D \xi\\
&+C\int_\R\left(1+|\xi|^{\alpha+2\gamma}\right) \E \left[\sup_{0\leq t\leq T}\left|\E\Big[\int_t^T\Gamma_{s,t}(\xi)\hat{f}(s,\xi)\D s \Big| \mathcal{F}_t\Big]\right|^2 \right] \D \xi.
\end{align*}
With the help of inequality~\eqref{eq1384} and Cauchy's inequality, we get
\begin{align*}
   \left| \E\Big[\int_t^T\Gamma_{s,t}(\xi)\hat{f}(s,\xi)\D s \Big| \mathcal{F}_t\Big] \right| &\leq \E\Big[\sup_{t\leq s\leq T}\bar{\Gamma}_{s,t} \int_t^T \hat{\Gamma}_{s,t}(\xi)|\hat{f}(s,\xi)| \D s\Big|\mathcal{F}_t\Big]\\
   &\leq \E\Big[\sup_{t\leq s\leq T}|\bar{\Gamma}_{s,t}|^4\Big|\mathcal{F}_t\Big]^{\frac{1}{4}} \E\Big[ \left( \int_t^T \hat{\Gamma}_{s,t}(\xi)|\hat{f}(s,\xi)| \D s \right)^\frac{4}{3}\Big|\mathcal{F}_t\Big]^{\frac{3}{4}}\\
   &\leq C \E\Big[ \left( \int_t^T \hat{\Gamma}_{s,t}(\xi) \D s \right)^\frac{2}{3} \left( \int_t^T \hat{\Gamma}_{s,t}(\xi)|\hat{f}(s,\xi)|^2 \D s \right)^\frac{2}{3} \Big|\mathcal{F}_t\Big]^{\frac{3}{4}}.
\end{align*}

Note that for all $p>0$, there exists constant $C_p$ such that
\begin{align}\label{eq1413}
     \int_0^T|\hat{\Gamma}_{s,t}(\xi)|^p\D t+ \int_0^T|\hat{\Gamma}_{s,t}(\xi)|^p\D s\leq T \wedge (C_p |\xi|^{-\alpha}).
\end{align}
This leads to
\begin{align*}
     \E\Big[\int_t^T\Gamma_{s,t}(\xi)\hat{f}(s,\xi)\D s \Big| \mathcal{F}_t\Big] &\leq C\left(1\wedge |\xi|^{-\frac{\alpha}{2}}\right)\E\Big[ \left( \int_t^T \hat{\Gamma}_{s,t}(\xi)|\hat{f}(s,\xi)|^2 \D s \right)^\frac{2}{3} \Big|\mathcal{F}_t\Big]^{\frac{3}{4}},
\end{align*}
which together with Doob's inequality, implies that
\begin{align}
&\int_\R\left(1+|\xi|^{\alpha+2\gamma}\right)\E\left[\sup_{0\leq t\leq T}\Big|\E\Big[\int_t^T\Gamma_{s,t}(\xi)\hat{f}(s,\xi)\D s \Big| \mathcal{F}_t\Big]\Big|^2\right]\D \xi\nonumber\\
\leq &C \int_\R\left(1+|\xi|^{2\gamma}\right)\E\left[\sup_{0\leq t\leq T}\Big|\E\Big[ \left( \int_0^T|\hat{f}(s,\xi)|^2\D s \right)^{\frac{2}{3}} \Big| \mathcal{F}_t\Big]\Big|^\frac{3}{2}\right]\D \xi\nonumber\\
\leq &C \,\int_\R\left(1+|\xi|^{2\gamma}\right)\E\left[\int_0^T|\hat{f}(s,\xi)|^2\D s \right]\D \xi\leq C\E\int_0^T\|f(s,\cdot)\|^2_{H^\gamma}\D s.\label{eq1421}
\end{align}
In the same spirit, we obtain
\begin{align*}
&\int_\R\left(1+|\xi|^{\alpha+2\gamma}\right)\E \left[ \sup_{0\leq t\leq T}\left|\E\left[\Gamma_{T,t}(\xi)\hat{g}(\xi) \big| \mathcal{F}_t\right]\right|^2 \right] \D \xi\\
\leq & C\,\E\int_{\R}(1+|\xi|^{\alpha+2\gamma})|\hat{g}(\xi)|^2\D \xi\leq C\, \E\left[\|g\|^2_{H^{\frac{\alpha}{2}+\gamma}}\right].
\end{align*}
This, together with~\eqref{eq1421}, proves the upper bound of $\E\left[\sup\limits_{t\in[0,T]}\|u(t,\cdot)\|^2_{H^{\frac{\alpha}{2}+\gamma}}\right]$. 

To prove the upper bound of $\E\left[\int_0^T\|u(t,\cdot)\|^2_{H^{\alpha+\gamma}}\D t\right]$, we notice that
\begin{align*}
    |\hat{u}(t,\xi)|^2&\leq 2\left|\E\left[{\Gamma}_{T,t}(\xi)\hat{g}(\xi)\big| \mathcal{F}_t\right]\right|^2+2\left|\E\left[\sup_{t\leq s\leq T}\bar{\Gamma}_{s,t} \int_t^T\hat{\Gamma}_{s,t}(\xi)|\hat{f}(s,\xi)|\D s\Big| \mathcal{F}_t\right] \right|^2\\
    &\leq 2\E\left[\left|\bar{\Gamma}_{T,t}\right|^2\big| \mathcal{F}_t\right]\cdot\E\left[\big|{\hat{\Gamma}}_{T,t}(\xi)\hat{g}(\xi)\big|^2\big| \mathcal{F}_t\right]\\
    &\quad+2\E\left[\big|\sup_{t\leq s\leq T}\bar{\Gamma}_{s,t}\big|^2\big| \mathcal{F}_t\right]\cdot\E\left[ \left( \int_t^T\hat{\Gamma}_{s,t}(\xi)|\hat{f}(s,\xi)|\D s \right)^2 \Big| \mathcal{F}_t\right] \\
    &\leq C\,\E\left[\big|{\hat{\Gamma}}_{T,t}(\xi)\hat{g}(\xi)\big|^2\big| \mathcal{F}_t\right]+C\,\E\left[ \left( \int_t^T\hat{\Gamma}_{s,t}(\xi)|\hat{f}(s,\xi)|\D s \right)^2 \Big| \mathcal{F}_t\right],
\end{align*}
where the third inequality comes from the fact~\eqref{eq1384}. Then
\begin{align*}
    &\E\left[\int_0^T\|u(t,\cdot)\|^2_{H^{\alpha+\gamma}}\D t\right] \leq C\,\E \left[ \int_0^T \int_{\R} \left( 1 + |\xi|^{2\alpha+2\gamma} \right) |\hat{u}(t,\xi)|^2 \D \xi\D t \right]\\
    \leq & C\,\int_0^T\int_{\R}\left(1+|\xi|^{2\alpha+2\gamma}\right)\E\left[\big|{\hat{\Gamma}}_{T,t}(\xi)\hat{g}(\xi)\big|^2+ \left(\int_t^T\hat{\Gamma}_{s,t}(\xi)|\hat{f}(s,\xi)|\D s \right)^2 \right]\D \xi\D t.
\end{align*}
Using the fact~\eqref{eq1413} again, we immediately get
\begin{align*}    &\int_0^T\int_{\R}\left(1+|\xi|^{2\alpha+2\gamma}\right)\E\left[\big|{\hat{\Gamma}}_{T,t}(\xi)\hat{g}(\xi)\big|^2\right]\D \xi\D t\\
\leq & \E\left[\sup_{\xi\in\R}\left\{\left(1+|\xi|^{\alpha}\right)\int_0^T |\hat{\Gamma}_{T,t}(\xi)|^2\D t\right\}\cdot \int_{\R}\left(1+|\xi|^{2\gamma+\alpha}\right)|\hat{g}(\xi)|^2\D \xi\right]
\leq C\,\E\left[\|g\|^2_{H^{\gamma+\frac{\alpha}{2}}}\right],
\end{align*}
and by Fubini's Theorem, we have
\begin{align*}
&\int_0^T\int_{\R}\left(1+|\xi|^{2\alpha+2\gamma}\right)\E\left[ \left( \int_t^T\hat{\Gamma}_{s,t}(\xi)|\hat{f}(s,\xi)|\D s \right)^2 \right]\D \xi\D t\\
\leq& \int_0^T \int_{\R}\left(1+|\xi|^{2\alpha+2\gamma}\right)\E\left[\int_t^T\hat{\Gamma}_{s,t}(\xi)\D s\int_t^T\hat{\Gamma}_{s,t}(\xi)\big|\hat{f}(s,\xi)\big|^2\D s\right]\D \xi\D t\\
\leq & \int_0^T \int_{\R}\left(1+|\xi|^{\alpha+2\gamma}\right)\E\left[\int_t^T\hat{\Gamma}_{s,t}(\xi)\big|\hat{f}(s,\xi)\big|^2\D s\right]\D \xi\D t\\
 =&\int_0^T \int_{\R}\left(1+|\xi|^{\alpha+2\gamma}\right)\E\left[\int_0^s\hat{\Gamma}_{s,t}(\xi)\D t\big|\hat{f}(s,\xi)\big|^2\right]\D \xi\D s\leq C\,\E\int_0^T \|f(s,\cdot)\|^2_{H^\gamma}\D s.
\end{align*}
Thus we finally complete our proof.
\end{proof}

\begin{theorem}
 Let Assumptions~\ref{Baisc assumption} and~\ref{Simple assumption} hold. If $(f,g)\in \mathbb{H}^\gamma \times L^2(\Omega,H^\gamma)$, then BSPDE~\eqref{linear-case} has a unique solution $(u,v)\in \left[ \mathcal{S}^2(0,T;H^{\frac{\alpha}{2}+\gamma})\cap\mathbb{H}^{\alpha+\gamma}(T) \right] \times \mathbb{H}^{\frac{\alpha}{2}+\gamma}(T)$ such that
    \begin{align*}
    &\E\left[\sup_{t\in[0,T]}\|u(t,\cdot)\|^2_{H^{\frac{\alpha}{2}+\gamma}}\right] + \E\left[ \int_0^T \|v(t,\cdot)\|^2_{H^{\frac{\alpha}{2}+\gamma}} \D t \right] + \E\left[ \int_0^T \|v(t,\cdot)\|^2_{H^{\frac{\alpha}{2}+\gamma}} \D t \right]\\
    \leq& C\,\E\left[\|g\|_{H^{\frac{\alpha}{2}+\gamma}}^2\right]+C\,\E \left[ \int_0^T\|f(t,\cdot)\|^2_{H^{\gamma}}\D t \right],
        \end{align*}
     where $\gamma>0$ and $C=C(\alpha,\beta,\gamma,K,m,T)$.  
\end{theorem}
\begin{proof}
    By the regularity of $(f,g)$, we know that for almost every $\xi\in\R$,
    \begin{align*}
       \E \left[ \int_0^T |\hat{f}(t,\xi)|^2 \D t \right] +\E\left[|g(\xi)|^2\right]<+\infty,
    \end{align*}
    which yields that BSDE~\eqref{eq2409021} has an unique solution $(\hat{u},\hat{v})$. By the standard BSDE estimates, we know that
    \begin{align*}
        \E\left[\int_0^T|\hat{v}(t,\xi)|^2\D t\right]\leq C\,\E\left[\int_0^T(1+|\xi|^\alpha)|\hat{u}(t,\xi)|^2+(1+|\xi|^\alpha)^{-1}|\hat{f}(t,\xi)|^2\D t\right]+C\,\E\left[|\hat{g}(\xi)|^2\right].
    \end{align*}
    Combing this with estimates in Lemma \ref{lemma1385}, we know that the Fourier inverse transformation of $(\hat{u},\hat{v})$ is well-defined and satisfies
       \begin{align*}
\E\left[\sup_{t\in[0,T]}\|u\|^2_{H^{\frac{\alpha}{2}+\gamma}}\right]+\E \left[ \int_0^T\left(\|u\|^2_{H^{\alpha+\gamma}}+\|v\|^2_{H^{\frac{\alpha}{2}+\gamma}}\right)\D t \right] \leq C\,\E\left[\|g\|_{H^{\frac{\alpha}{2}+\gamma}}^2\right]+C\,\E \left[ \int_0^T\|f\|^2_{H^{\gamma}}\D t \right].
        \end{align*}
        
To prove that $(u,v)$ is the (strong) solution to BSPDE~\eqref{linear-case}, we need to prove $u$ is continuous at the boundary $t=T$, i.e.,
\begin{align*}
    \lim_{t\rightarrow T}\E\left[\|u(t,x)-g(x)\|_{H^{\frac{\alpha}{2}+\gamma}}\right]=0,
\end{align*}
which is equivalent to
\begin{align}\label{eq1487}
     \lim_{t\rightarrow T}\int_{\R}(1+|\xi|^{\alpha+2\gamma})\E\left[\bigg|\E\Big[\Gamma_{T,t}(\xi)\hat{g}(\xi)+\int_t^T\Gamma_{s,t}(\xi)\hat{f}(s,\xi)\D s \Big| \mathcal{F}_t\Big]-\hat{g}(\xi)\bigg|^2\right]\D \xi=0.
\end{align}
Notice that
\begin{align*}
&\frac{1}{3}\E\left[\bigg|\E\Big[\Gamma_{T,t}(\xi)\hat{g}(\xi)+\int_t^T\Gamma_{s,t}(\xi)\hat{f}(s,\xi)\D s \Big| \mathcal{F}_t\Big]-\hat{g}(\xi)\bigg|^2\right]\\
\leq &
\E\left[\bigg|\E\Big[\left(\Gamma_{T,t}(\xi)-1\right)\hat{g}(\xi)\Big| \mathcal{F}_t\Big]\bigg|^2\right]+\E\left[\bigg|\E\Big[\hat{g}(\xi)\Big| \mathcal{F}_t\Big]-\hat{g}(\xi)\bigg|^2\right] + \E \left[ \bigg| \E \Big[ \int_t^T \Gamma_{s,t}(\xi) \hat{f}(s,\xi)\D s \Big| \mathcal{F}_t\Big]\bigg|^2\right].
\end{align*}
Then~\eqref{eq1487} can be shown by dominated convergence theorem.
\end{proof}

\subsection{ Preliminary properties of fractional Laplacian}
Similar to classical Hölder spaces of finite-dimensional vector-valued functions, we have the following propositions about how fractional Laplacian interacts with $C^{\alpha}$-norms. The proof can be completed in the same spirit as classical ones (see e.g. proposition 2.7 in \cite{silvestre2007regularity}), so we omitted it here.
\begin{lemma}\label{delta Hölder}
Let $u\in C^{\alpha+\beta}({\mathbb{\R},\mathcal{L}^2(0,T)})$. For each $\alpha\in(0,+\infty)$ and $\beta\in(0,+\infty)\backslash \mathbb{N}$, we have $(-\Delta)^{\frac{\alpha}{2}}u\in  C^{\beta}({\mathbb{\R},\mathcal{L}^2(0,T)})$ and there exists $C=C(\alpha,\beta)$ such that
\begin{align*}
      \|(-\Delta)^{\frac{\alpha}{2}}u\|_{\beta,\mL^2}\leq C\|u\|_{\alpha+\beta,\mL^2} .
\end{align*}
\end{lemma}

To study the properties of $G_{t,x}(x)$, we define
\begin{align}\label{defG}
    G(x)=\frac{1}{2\pi} \int_{\mathbb{R}}\exp\Big(-i\xi x-|\xi|^{\alpha}\Big)\D \xi.
\end{align}
Its derivative $D^k G(x)$ plays an important role in our proof and has the following polynomial decay property, which indicates $D^k G(x)$ is integrable. 
\begin{lemma}\label{Lemma G bound}
For each $ k\in \mathbb{N}$ and $\gamma>0$, there exists $C=C(\alpha,k,\gamma)$ such that
\begin{equation}\label{G bound} 
    \left|D^k G(x)\right|\leq \frac{C}{1+|x|^{1+\alpha+k}},
\end{equation}
and
\begin{equation}\label{delta G bound} 
\left|(-\Delta)^{\frac{\gamma}{2}}G(x)\right|\leq \frac{C}{1+|x|^{1+\gamma}}.
\end{equation}
\end{lemma}

\begin{proof}
The inequality~\eqref{G bound} has been proved in Corollary 1,\cite{debbi2005solutions}. It is sufficient to prove the~\eqref{delta G bound} when $x>0$. By definition of $ (-\Delta)^{\frac{\gamma}{2}}G(x)$, we have
    \begin{align*}
        (-\Delta)^{\frac{\gamma}{2}}G(x)=&\frac{1}{2\pi}\int_{\R} |\xi|^{\gamma}\exp\left(-i\xi x-|\xi|^{\alpha}\right)\D \xi= \frac{1}{\pi}\mathrm{Re}\left\{\int_{0}^{+\infty} \xi^{\gamma}\exp\left(-i\xi x-\xi^{\alpha}\right)\D \xi\right\}.
    \end{align*}
Let $0<r\leq R<\infty$, and define the curve 
$$
C=[r,R]\vee \left\{Re^{-i\theta}:0
\leq \theta\leq  \frac{\pi}{2\alpha}\right\}\vee \{\lambda e^{-i\frac{\pi}{2\alpha}}:r\leq\lambda\leq R\}\vee \left\{re^{-i\theta}:0
\leq \theta\leq  \frac{\pi}{2\alpha}\right\}^*,
$$
where the symbol $\vee$ means "followed by" and $*$ means that the curve is taken in the opposite direction. Applying Cauchy Theorem to function $f(z)=z^\gamma\exp(-izx-z^\alpha)$ on the curve C, as well as letting $r$ tend to zero and $R$ tend to infinity, we obtain that
$$
\int_{0}^{+\infty} \xi^{\gamma}\exp\left(-i\xi x-\xi^{\alpha}\right)\D \xi=\int_{0}^{+\infty} e^{-i\frac{(\gamma+1)\pi}{2\alpha}}\xi^{\gamma}\exp\left(-\xi x e^{i\frac{(\alpha-1)\pi}{2\alpha}}-\xi^{\alpha}e^{-i\frac{\pi}{2}}\right)\D \xi.
$$
 Making the change of variable $\lambda=\xi x$, we find
\begin{align*}
        (-\Delta)^{\frac{\gamma}{2}}G(x)=& \frac{1}{\pi}\mathrm{Re}\left\{\int_{0}^{+\infty} e^{-i\frac{(\gamma+1)\pi}{2\alpha}}\xi^{\gamma}\exp\left(-\xi x e^{i\frac{(\alpha-1)\pi}{2\alpha}}-\xi^{\alpha}e^{-i\frac{\pi}{2}}\right)\D \xi\right\}\\
        = &\frac{1}{\pi x^{1+\gamma}}\mathrm{Re}\left\{\int_{0}^{+\infty} e^{-i\frac{(\gamma+1)\pi}{2\alpha}}\lambda^{\gamma}\exp\left(-\lambda e^{i\frac{(\alpha-1)\pi}{2\alpha}}-x^{-\alpha}\lambda^{\alpha}e^{-i\frac{\pi}{2}}\right)\D \lambda\right\}\\
        \leq &\frac{1}{\pi x^{1+\gamma}}\int_{0}^{+\infty}\left| e^{-i\frac{(\gamma+1)\pi}{2\alpha}}\lambda^{\gamma}\exp\left(-\lambda e^{i\frac{(\alpha-1)\pi}{2\alpha}}-x^{-\alpha}\lambda^{\alpha}e^{-i\frac{\pi}{2}}\right)\right|\D \lambda\\
        = & \frac{1}{\pi x^{1+\gamma}}\int_{0}^{+\infty} \lambda^{\gamma} e^{ -\lambda \cos \frac{(\alpha-1)\pi}{2\alpha} }\D \lambda=\frac{C}{x^{1+\gamma}}.
    \end{align*}
On the other hand, \begin{align*}
     \left|(-\Delta)^{\frac{\gamma}{2}}G(x)\right| \leq \frac{1}{\pi}\int_{0}^{+\infty} \left|\xi^{\gamma}\exp\left(-i\xi x-\xi^{\alpha}\right)\right|\D \xi\leq\frac{1}{\pi}\int_{0}^{+\infty} \xi^{\gamma}\exp\left(-\xi^{\alpha}\right)\D \xi=C.
\end{align*}
Therefore 
$$
  \left|(-\Delta)^{\frac{\gamma}{2}}G(x)\right|\leq \frac{C}{|x|^{1+\gamma}}\wedge C ,
$$
and we complete our proof.
\end{proof}
\begin{remark}
    In \cite{debbi2005solutions}, authors actually prove the the inequality~\eqref{G bound} holds for all $\alpha\in(1,+\infty)\backslash \mathbb{N}$. Here we give the the upper bound of $(\Delta)^{\gamma}G$. Although the order of estimates~\eqref{delta G bound} may not be optimal, it is enough in the proof of our article.
\end{remark}

Similar to the heat kernel of classical Laplacian, we can check that $G_{t,s}(x)$ satisfies the following semigroup property. For more details, readers can refer to Lemma 1 in \cite{debbi2005solutions}. 

\begin{lemma}\label{semi-group}
    \begin{enumerate}[(i)]
        \item \label{int is 1}For all $0\leq s<t\leq T$, $x\in {\R}$, we have
       $ G_{t,s}(x)>0 $ and 
            $$\int_{\mathbb{R}}G_{t,s}(x)\D x=1.$$
        \item Denote $ R_s^t f(x) =\int_{\mathbb{R}} G_{t,s}(x-y)f(y)\D y$ with $0\leq s<t\leq T$. Then it holds that
        \begin{align*}
            R_{t_2}^{t_3}\left(R_{t_1}^{t_2} f\right)=R_{t_1}^{t_3}f,
        \end{align*}
        where $0\leq t_1<t_2<t_3\leq T$.
        \item \label{variable change}If $a(r)>0$, $r\in[0,T]$, define 
    \begin{align}\label{defA}
        A_{t,s}=\int_s^ta(r)\D r>0,
    \end{align}
 then for all $0\leq s<t\leq T$, it holds that
$$
D^k G_{t,s}(x)= A_{t,s}^{-\frac{1+k}{\alpha}} D^kG(z)\Big|_{z=A_{t,s}^{-\frac{1}{\alpha}}x}, \quad\quad \forall\,k\in \mathbb{N},
$$
and
$$
(-\Delta)^{\frac{\gamma}{2}}G_{t,s}(x)= A_{t,s}^{-\frac{1+\gamma}{\alpha}}(-\Delta)^{\frac{\gamma}{2}}G(z)\Big|_{z=A_{t,s}^{-\frac{1}{\alpha}}x}, \quad\quad \forall\, \gamma>0.
$$
\end{enumerate}
\end{lemma}
\begin{remark}
    For $a(\cdot)\equiv 1$ and $\alpha\in (1,2]$, $P(t,s;x,y):=G_{t,s}(y-x)$ is the transition density of the $\alpha$-stable process. Some properties such as the positivity of G are studied in \cite{Blumenthal1990}. In general, the fractional heat kernel may not be positive when $\alpha>2$. 
\end{remark}

\begin{remark}By Lemma \ref{semi-group}~\eqref{int is 1}, we get
    \begin{align}
  \label{eq 174}   D^k\int_{\R}G_{t,s}(x-y)\D y&=\int_{\R}D^k G_{t,s}(x-y)\D y=0, \quad\quad \forall\, k\geq 1,
\end{align}
  and
  \begin{align}
  \label{eq 175}  (-\Delta)^{\frac{\gamma}{2}}\int_{\R}G_{t,s}(x-y)\D y&=\int_{\R}(-\Delta)^{\frac{\gamma}{2}}G_{t,s}(x-y)\D y=0,\quad\quad  \forall\, \gamma\geq 1.
    \end{align}
\end{remark}

With the help of Lemma \ref{G bound}, we have the upper bound estimates for $G_{t,x}(x)$ as follows.
\begin{lemma}\label{all G estimate}
Let $M>0$ be the constant such that $a(\cdot) \in\left[\frac{1}{M},M\right]$. For $0\leq s<t\leq T$, we have the following estimates.
\begin{enumerate}[(i)]
\item \label{G_t bound}
For each $x\in\R$ and $k\in \mathbb{N}$, there exists $C=C(\alpha,k)$ such that
\begin{align*}
    \left|D^k G_{t,s}(x)\right|\leq \frac{C(t-s)^{-\frac{1+k}{\alpha}}M^{\frac{1+k}{\alpha}}}{M^{-\frac{k+\alpha+1}{\alpha}}(t-s)^{-\frac{k+\alpha+1}{\alpha}}|x|^{k+\alpha+1}+1}.
\end{align*}
    \item  \label{sup G estimate} 
    For each $\gamma\in[0,\alpha)$, there is $C=C(\alpha,\gamma,M,T)$ such that
            \begin{align*}
       \int_{\R} \sup_{t,s\in[0,T]}G_{t,s}(x)|x|^{\gamma}\D x\leq C.
        \end{align*}
        
        \item  \label{sup G_2 estimate}
        For each $k\geq 2$, $\gamma\geq0$ and $\eta>0$, there exists $C=C(k,M,\gamma,T)$ such that
    \begin{align*}
        \left\{\begin{aligned}
       \int_{|x|\leq \eta}\sup_{t\in[0,T]} \int_t^T\left|D^k G_{s,t}(x)\right||x|^{\gamma}\D s\D x\leq C\eta^{\alpha+\gamma-k},\quad k<\alpha+\gamma,\\
       \int_{|x|> \eta}\sup_{t\in[0,T]} \int_t^T\left|D^k G_{s,t}(x)\right||x|^{\gamma}\D s\D x\leq C\eta^{\alpha+\gamma-k} ,\quad k\geq \alpha+\gamma.
        \end{aligned}\right.
    \end{align*}

    \item\label{delta G sup int bound}    For each $\gamma\in[0,\alpha)$, there is $C=C(\alpha,\gamma,M)>0$ such that
   \begin{align*}  \int_{\R}\sup_{r\in[t,T]}\int_{r_t(\varepsilon)}^r\left|(-\Delta)^{\frac{\alpha}{2}} G_{r,u}(x)\right||x|^{\gamma}\D u\D x\leq C \varepsilon^{\frac{\gamma}{\alpha}}.
    \end{align*}
    
    \item \label{Lemma G 195}  For each $\gamma\in[0,\alpha)$, there exists $C=C(\alpha,\gamma,M)>0$ such that  \begin{align*}
        \int_{\R}\big|(-\Delta)^{\frac{\alpha}{2}} G_{t,s}(x)\big||x|^{\gamma}\D x\leq C(t-s)^{\frac{\gamma}{\alpha}-1}.
        \end{align*}

         \item \label{Lemma G 200}  For each $k\in\mathbb{N}$ and $ \gamma \in [0,\alpha+k)$, there is $C=C(\alpha,\gamma,M,k)>0$ such that  \begin{align*}
        \int_{\R}\big|D^k G_{t,s}(x)\big||x|^{\gamma}\D x\leq C(t-s)^{\frac{\gamma-k}{\alpha}}.
        \end{align*}

       \item \label{Lemma G 219}For each $l>1$ and $k\in\mathbb{N}_+$, there exists $C=C(\alpha,l,T,\gamma,M,k)>0$ such that
       \begin{align*}
       \left\{\begin{aligned}
          \int_0^T \Big(\int_{|x|>\eta}\left| D^k G_{T,t}(x)\right||x|^{\gamma}\D x\Big)^l\D t\leq C \eta^{(\gamma-k)l+\alpha}, \quad k> \gamma+\frac{\alpha}{l},\\[0.3cm]
           \int_0^T \Big(\int_{|x|\leq \eta}\left| D^k G_{T,t}(x)\right||x|^{\gamma}\D x\Big)^l\D t\leq C \eta^{(\gamma-k)l+\alpha}, \quad k< \gamma+\frac{\alpha}{l}.
             \end{aligned}\right.
        \end{align*}
    \end{enumerate}
\end{lemma}  

\begin{proof}[Proof of Lemma \ref{all G estimate}~\eqref{G_t bound}]
    The proof is completed by Lemma \ref{G bound} and Lemma \ref{semi-group}~\eqref{variable change}.
\end{proof}
\begin{proof}[Proof of Lemma \ref{all G estimate}~\eqref{sup G estimate}]
By  Lemma \ref{all G estimate}~\eqref{G_t bound},
  \begin{align*}
     \sup_{0\leq s<t\leq T}G_{t,s}(x) |x|^{\gamma} &\leq \frac{C|x|^{\gamma}}{1 + |x|^{\alpha+1}} \leq \begin{cases}
          C|x|^{\gamma-1},\quad\quad\quad|x|\leq 1,\\
        C|x|^{\gamma-\alpha-1},\quad\quad|x|>1.
     \end{cases}
  \end{align*}
 Then the proof is completed by integrating on $x$. 
\end{proof}

\begin{proof}[Proof of Lemma \ref{all G estimate}~\eqref{sup G_2 estimate}]
In view of Lemma \ref{all G estimate}~\eqref{G_t bound},
    \begin{align*}
        \sup_{t\in[0,T]}\int_t^T\left|D^kG_{s,t}(x)\right| |x|^\gamma \D s \leq& C \sup_{t\in[0,T]}\int_t^T\left|\frac{M^{\frac{1+k}{\alpha}}}{(s-t)^{\frac{1+k}{\alpha}}+M^{-\frac{k+\alpha+1}{\alpha}}(s-t)^{-1}|x|^{k+\alpha+1}}\right|\D s \cdot |x|^\gamma \\
        \leq & C \int_0^T \frac{M^{\frac{1+k}{\alpha}} v }{v^{\frac{\alpha+1+k}{\alpha}}+M^{-\frac{k+\alpha+1}{\alpha}}|x|^{k+\alpha+1}} \D v \cdot |x|^\gamma \\
        \xlongequal[]{v = \frac{|x|^\alpha}{r^\alpha}} & C \int^{+\infty}_{|x|/T^\frac{1}{\alpha}} \frac{M^{\frac{1+k}{\alpha}} |x|^\alpha/r^\alpha}{|x|^{k+\alpha+1}\left( r^{-k-\alpha-1} + M^{-\frac{k+\alpha+1}{\alpha}} \right)} \frac{\alpha|x|^\alpha}{r^{\alpha+1}} \D r \cdot |x|^\gamma \\
        \leq & C \int^{+\infty}_0 \alpha M^{\frac{1+k}{\alpha}}  \frac{ |x|^{\alpha-k-1} r^{k-\alpha} }{1 + M^{-\frac{k+\alpha+1}{\alpha}} r^{\alpha+k+1}} \D r \cdot |x|^\gamma 
        \leq C|x|^{\alpha+\gamma-k-1}.
    \end{align*}
     Then the proof is completed by integrating on $x$. 
\end{proof}

\begin{proof}[Proof of Lemma \ref{all G estimate}~\eqref{delta G sup int bound}]

By Lemma \ref{semi-group} and Lemma~\ref{Lemma G bound}, we have
\begin{align*}
    &\int_{r_t(\varepsilon)}^r\left|(-\Delta)^{\frac{\alpha}{2}} G_{r,u}(x)\right|\D u= \int_{r_t(\varepsilon)}^r A_{r,u}^{-\frac{\alpha+1}{\alpha}}\left|(-\Delta)^{\frac{\alpha}{2}} G(x\cdot A_{r,u}^{-\frac{1}{\alpha}})\right|\D u\\
   \leq& \int_{r_t(\varepsilon)}^r \frac{C}{A_{r,u}^{\frac{\alpha+1}{\alpha}}+|x|^{\alpha+1}}\D u \leq \int_{r_t(\varepsilon)}^r \frac{C}{M^{-\frac{\alpha+1}{\alpha}}{{(r-u)}^{\frac{\alpha+1}{\alpha}}+|x|^{\alpha+1}}}\D u.
\end{align*}
It yields that
\begin{align*}
     &\int_{\R}\sup_{r\in[t,T]}\int_{r_t(\varepsilon)}^r\left|(-\Delta)^{\frac{\alpha}{2}} G_{r,u}(x)\right||x|^{\gamma}\D u\D x\\
     \leq &\int_0^\varepsilon \int_{\R} \frac{C|x|^{\gamma}}{M^{-\frac{\alpha+1}{\alpha}}s^{\frac{\alpha+1}{\alpha}}+|x|^{\alpha+1}}\D x\D s\leq \int_0^\varepsilon C s^{\frac{\gamma-\alpha}{\alpha}}\D s= C \varepsilon^{\frac{\gamma}{\alpha}}.
\end{align*}
\end{proof}

\begin{proof}[Proof of Lemma \ref{all G estimate}~\eqref{Lemma G 195} and~\eqref{Lemma G 200}]
In view of Lemma \ref{G bound} and Lemma \ref{semi-group}, we get that
\begin{align*}
    \int_{\R} \left|(-\Delta)^{\frac{\alpha}{2}}G_{t,s}(x)\right||x|^{\gamma}\D x & =A_{t,s}^{\frac{\gamma}{\alpha}-1}\int_{\R} \left|(-\Delta)^{\frac{\alpha}{2}}G(z)\right||z|^{\gamma}\D z \leq  C (t-s)^{\frac{\gamma}{\alpha}-1}\int_{\R} \frac{|x|^{\gamma}}{1+|x|^{1+\alpha}}\D x,
\end{align*}
and
\begin{align*}
    \int_{\R} \left|D^k G_{t,s}(x)\right||x|^{\gamma}\D x & =A_{t,s}^{\frac{\gamma-k}{\alpha}}\int_{\R} \left|D^k G(z)\right||z|^{\gamma}\D z \leq  C (t-s)^{\frac{\gamma-k}{\alpha}}\int_{\R} \frac{|x|^{\gamma}}{1+|x|^{1+\alpha+k}}\D x.
\end{align*}
\end{proof}

\begin{proof}[Proof of Lemma \ref{all G estimate}~\eqref{Lemma G 219}]
First, we prove the case of $\gamma<k-\frac{\alpha}{l}$. By Lemma \ref{all G estimate}~\eqref{G_t bound}, we have
    \begin{align*}
         & \int_0^T \Big(\int_{|x|>\eta}\left| D^k G_{T,t}(x)\right||x|^{\gamma}\D x\Big)^l\D t\\
    \leq & C\int_0^T \Big(\int_{|x|>\eta} \frac{(T-t) M^{\frac{k+1}{\alpha}} |x|^\gamma }{(T-t)^{\frac{k+1+\alpha}{\alpha}}+M^{-\frac{k+1+\alpha}{\alpha}}|x|^{k+1+\alpha}}\D x\Big)^l\D t\\
     =& C' \int_0^T \Big(\int_{|x|>\eta} \frac{t|x|^{\gamma}}{(Mt)^{\frac{k+1+\alpha}{\alpha}}+|x|^{k+1+\alpha}}\D x\Big)^l\D t\\
     =& C'' \int_0^T \Big(\int_{ |y|> \eta {(Mt)^{-\frac{1}{\alpha}}}} \frac{|y|^{\gamma}}{1+|y|^{k+1+\alpha}}\D y\Big)^lt^{\frac{(\gamma-k)l}{\alpha}}\D t.
\end{align*}
If $\eta^{\alpha}<MT$, it holds that
\begin{align*}
&\int_0^T \Big(\int_{|y|>{\eta(Mt)^{-\frac{1}{\alpha}}}} \frac{|y|^{\gamma}}{1+|y|^{k+1+\alpha}}\D x\Big)^lt^{\frac{(\gamma-k)l}{\alpha}}\D t\\
\leq & \int_0^{M^{-1}\eta^{\alpha}} \Big(\int_{|y|> \eta (Mt)^{-\frac{1}{\alpha}} }|y|^{\gamma-k-1-\alpha}\D y\Big)^l t^{\frac{(\gamma-k)l}{\alpha}} \D t+\int_{M^{-1}\eta^{\alpha}}^T \Big(\int_{|y|> 1}|y|^{\gamma-k-1-\alpha}\D y\Big)^l t^{\frac{(\gamma-k)l}{\alpha}} \D t\\
&+\int_{M^{-1}\eta^{\alpha}}^T \Big(\int_{(Mt)^{-\frac{1}{\alpha}}\eta<|y|\leq 1}|y|^{\gamma}\D y\Big)^l t^{\frac{(\gamma-k)l}{\alpha}} \D t \leq C\eta^{\alpha+(\gamma-k)l}.
\end{align*} 
If $\eta^{\alpha}\geq MT$, it holds that
\begin{align*}
&\int_0^T \Big(\int_{|y|>{\eta(Mt)^{-\frac{1}{\alpha}}}} \frac{|y|^{\gamma}}{1+|y|^{k+1+\alpha}}\D x\Big)^lt^{\frac{(\gamma-k)l}{\alpha}}\D t\\
\leq & \int_0^{M^{-1}\eta^{\alpha}} \Big(\int_{|y|> 1 }|y|^{\gamma-k-1-\alpha}\D y\Big)^l t^{\frac{(\gamma-k)l}{\alpha}} \D t \leq C \eta^{\alpha+(\gamma-k)l}.
\end{align*}

The proof of the case of $\gamma>k-\frac{\alpha}{l}$ can be completed in a similar way. Notice that
    \begin{align*}
          \int_0^T \Big(\int_{|x| \leq \eta}\left| D^k G_{T,t}(x)\right||x|^{\gamma}\D x\Big)^l\D t
     \leq  C \int_0^T \Big(\int_{ |y| \leq \eta {(Mt)^{-\frac{1}{\alpha}}}} \frac{|y|^{\gamma}}{1+|y|^{k+1+\alpha}}\D y\Big)^lt^{\frac{(\gamma-k)l}{\alpha}}\D t.
\end{align*}
If $\eta^{\alpha}<MT$, it holds that
\begin{align*}
&\int_0^T \Big(\int_{|y| \leq {\eta(Mt)^{-\frac{1}{\alpha}}}} \frac{|y|^{\gamma}}{1+|y|^{k+1+\alpha}}\D x\Big)^lt^{\frac{(\gamma-k)l}{\alpha}}\D t\\
    \leq & \int_0^{M^{-1}\eta^{\alpha}} \Big(\int_{1<|y| \leq \eta (Mt)^{-\frac{1}{\alpha}} }|y|^{\gamma-k-1-\alpha}\D y\Big)^l t^{\frac{(\gamma-k)l}{\alpha}} \D t+\int_0^{M^{-1}\eta^{\alpha}} \Big(\int_{|y|\leq  1}|y|^{\gamma}\D y\Big)^l t^{\frac{(\gamma-k)l}{\alpha}} \D t\\
&+\int_{M^{-1}\eta^{\alpha}}^T \Big(\int_{|y| \leq (Mt)^{-\frac{1}{\alpha}}\eta}|y|^{\gamma}\D y\Big)^l t^{\frac{(\gamma-k)l}{\alpha}} \D t \leq C\eta^{\alpha+(\gamma-k)l}.
\end{align*} 
If $\eta^{\alpha}\geq MT$, it holds that
\begin{align*}
&\int_0^T \Big(\int_{|y| \leq {\eta(Mt)^{-\frac{1}{\alpha}}}} \frac{|y|^{\gamma}}{1+|y|^{k+1+\alpha}}\D x\Big)^lt^{\frac{(\gamma-k)l}{\alpha}}\D t\\
\leq & \int_0^{M^{-1}\eta^\alpha} \Big(\int_{|y|\leq  1 }|y|^{\gamma}\D y\Big)^l t^{\frac{(\gamma-k)l}{\alpha}} \D t+\int_0^{M^{-1}\eta^\alpha} \Big(\int_{1<|y|\leq \eta(Mt)^{-\frac{1}{\alpha}} } |y|^{\gamma-k-1-\alpha}\D y\Big)^l t^{\frac{(\gamma-k)l}{\alpha}} \D t \\
\leq & C\eta^{\alpha+(\gamma-k)l}.
\end{align*}
\end{proof}

\subsection{Proof of Lemma \ref{uniqueness}}
\begin{proof} 
First, we show that
\begin{align}\label{eq 817}
    u(t,x)=R^T_t g(x)+\int_t^T R_t^s f(s)(x)\D s-\int_t^T R_t^s v(s)(x)\D \widetilde{W}_s.
\end{align}
    For fixed $t\in [0,T)$, applying Ito's formula to $G_{s,t}(x-y)u(s,y)$ with $s\in (t,T]$, it holds that
    \begin{align*}
       \D G_{s,t}(x-y)u(s,y)=& -a(s)(-\Delta)^{\frac{\alpha}{2}}G_{s,t}(x-y)u(s,y)\D s+G_{s,t}(x-y)v(s,y)\D W_s\\
       &+G_{s,t}(x-y)\left[a(s)(-\Delta)^{\frac{\alpha}{2}}u(s,y)-f(s,y)-\sigma(s) v(s,y)\right]\D s.
    \end{align*}
    For each $r\in (t,T]$, integrating with $s$ from $r$ to $T$ and with $y$ on $\R$, by~\eqref{eq202407101} we get 
    \begin{align*}
        R^r_t u(r)(x)
    = R_t^T g(x)+\int_r^T \left[ R_t^s f(s)(x)+\sigma (s) R_t^s v(s)(x) \right] \D s-\int_r^T R_t^s v(s)(x)\D W_s.
    \end{align*}
    To show~\eqref{eq 817}, it suffices to prove when $r$ tends to $t$, it holds that $\mathbb{P}$-a.s.,
    \begin{align}
   \label{lim rt1'}    & \lim_{r\rightarrow t} R^r_t u(r)(x)= u(t,x),\\
   \label{lim rt2'}       & \lim_{r\rightarrow t} \int_t^r \left[ R_t^s f(s)(x)+\sigma (s) R_t^s v(s)(x) \right] \D s=0, \\
    \label{lim rt3'}     & \lim_{r\rightarrow t} \int_t^r R_t^s v(s)(x)\D W_s =0.
    \end{align}

In order to show the equation (\ref{lim rt1'}), we notice that
\begin{align}
  &\E\Big[ \big|\int_{\mathbb{R}} G_{r,t}(x-y)u(r,y)\D y -u(t,x)\big|^2\Big]=  \E\Big[ \big|\int_{\mathbb{R}}G_{r,t}(x-y)(u(r,y)-u(t,x))\D y \big|^2\Big]\nonumber\\
 \leq& \int_{\mathbb{R}} G_{r,t}(x-y) \E\left[\big|u(r,y)-u(t,x)\big|^2\right]\D y  = \int_{\mathbb{R}} G(z) \E\left[\big|u(r,x-A_{r,t}^{\frac{1}{\alpha}}z)-u(t,x)\big|^2\right]\D z,
\end{align}
where the first inequality uses Cauchy inequality and the second equality follows from Lemma \ref{semi-group}. In view of Lemma~\ref{Baisc assumption}, we have $u\in C^{\beta}\left(\mathbb{R},\mathcal{S}^2_\mathbb{F}(0, T)\right)$. By Lebesgue's dominated convergence theorem, we obtain
\begin{align*}
    \lim_{r\rightarrow t}\int_{\mathbb{R}} G(z) \E\left[\big|u(r,x-A_{r,t}^{\frac{1}{\alpha}}z)-u(t,x)\big|^2\right]\D z=0,
\end{align*}
which proves (\ref{lim rt1'}). 

Next we prove (\ref{lim rt2'}) and (\ref{lim rt3'}). We only prove~\eqref{lim rt3'} and the proof of~\eqref{lim rt2'} is similar. By BDG inequality, it follows
\begin{align*}
     \E\Big[ \big|\int_t^r R_t^s v(s)(x)\D W_s\big|^2\Big] \leq  C \int_t^r \E\Big[ \big| \int_{\mathbb{R}}G_{s,t}(x-y) v(s,y)\D y\big|^2\Big]\D s.
\end{align*}
Similarly, we have
\begin{align*}
    \int_t^r \E\Big[ \big| \int_{\mathbb{R}}G_{s,t}(x-y) v(s,y)\D y\big|^2\Big]\D s&\leq    \int_{\mathbb{R}}G(z)  \int_t^r\E\Big[\big| v(s,x-A_{r,t}^{\frac{1}{\alpha}}z)\big|^2\Big]\D s\D z.
\end{align*}
In view Lemma~\ref{Baisc assumption}, we have $v\in C^{\beta}\left(\mathbb{R},\mathcal{L}^2_\mathbb{F}(0, T)\right)$. Then, we complete the proof of (\ref{lim rt3'}) by Lebesgue's dominated convergence theorem again. Thus we prove~\eqref{eq 817}.

By~\eqref{eq 817} and~\eqref{eq 87}, it holds that
    \begin{align}
   \nonumber     &u(t,x)+\int_t^T R_t^s v(s)(x)\D \widetilde{W}_s=R^T_t g(x)+\int_t^T R_t^s f(s)(x)\D s\\
     =& R^T_t p(t;x)+  R^T_t\Big(\int_t^T q(s)(x)\mathrm{d}\widetilde{W}_s\Big)+ \int_t^T R_t^s Y(t;s)(x)\D s+\int_t^TR_t^s\int_t^{s} Z(r;s)(x)\D \widetilde{W}_r\D s \label{eq 167}.
    \end{align}
By Lemma \ref{semi-group}, it holds that
\begin{align*}
   R^T_t\Big(\int_t^T q(s)(x)\mathrm{d}\widetilde{W}_s\Big)=\int_t^T R_t^s \big(R_s^T q(s)\big)(x)\D \widetilde{W}_s.
\end{align*}
In view of the stochastic Fubini theorem and Lemma \ref{semi-group}, we have
\begin{align*}
   \int_t^T R_t^s \int_t^s Z(r;s)(x)\D \widetilde{W}_r\D s &= \int_t^T \int_s^T R_t^r Z(s;r)(x)\D r\D\widetilde{W}_s=\int_t^T R_t^s \int_s^T R_s^r Z(s;r)(x)\D r\D\widetilde{W}_s.
\end{align*}
 Taking the conditional expectation on both sides of (\ref{eq 167}), we get 
 $$
 u(t,x)= R^T_t p(t)(x)+ \int_t^T R_t^s Y(t;s)(x)\D s.
 $$
 It follows that
 $$
 \int_t^T R_t^s \Big( v(s) - R_s^T q(s) - \int_s^T R_s^r Z(s;r)\D r \Big)(x)\D \widetilde{W}_s = 0.
 $$
Because of the arbitrariness of $(t,x)$, we complete our proof. 
\end{proof}

\bibliographystyle{abbrv}
\bibliography{ref}
\end{document}